\documentclass[a4paper,12pt]{amsart}

\usepackage{amsmath}
\usepackage[latin1]{inputenc}
\usepackage{amsfonts}
\usepackage{amssymb}
\usepackage{geometry}
\usepackage{xcolor}
\newtheorem{thm}{Theorem}[section]
\newtheorem{cor}[thm]{Corollary}
\newtheorem{lem}[thm]{Lemma}

\newtheorem{defn}[thm]{Definition}
\newtheorem{rem}[thm]{Remark}

\numberwithin{equation}{section}

\def\N{\mathbb N}
\def\Z{\mathbb Z}
\def\R{\mathbb R}

\def\b{\mathrm{b}}
\def\a{a}
\def\B{B}

\newcommand{\mand}{\quad\text{and}\quad}

\title[Reiteration Theorem for ${\mathcal R}$ and ${\mathcal L}$-space]
{Reiteration Theorem for ${\mathcal R}$ and ${\mathcal L}$-spaces with the same parameter.}

\author[Doktorski, Fern\'andez-Mart\'{\i}nez and Signes]{Leo R. Ya. Doktorski, P. Fern\'andez-Mart\'{\i}nez, T. Signes}
\address[Leo R. Ya. Doktorski]{Department Object Recognition, Fraunhofer Institute of Optronics, System Technologies and Image Exploitation IOSB, Gutleuthausstr. 1, 76275 Ettlingen, Germany.}
\address[Pedro Fern\'andez-Mart\'{\i}nez and Teresa M. Signes]{Departamento de Matem\'aticas \\
Facultad de Matem\'aticas \\ Universidad de Murcia \\ Campus de
Espinardo \\ 30071 Espinardo (Murcia), Spain.}

\date{\today}

\begin{document}
%\sloppy 
\begin{abstract}
Let $E, F, E_0, E_1$ be rearrangement invariant spaces; let $a, \b, \b_0, \b_1$ be slowly varying functions and $0< \theta_0,\theta_1<1$. We characterize the interpolation spaces
$$\Big(\overline{X}^{\mathcal R}_{\theta_0,\b_0,E_0,\a,F}, \overline{X}^{\mathcal L}_{\theta_1,\b_1,E_1,\a,F}\Big)_{\eta,\b,E}\:, \quad 0\leq\eta\leq1,$$
when the parameters $\theta_0$ and $\theta_1$ are equal (under appropriate conditions on $\b_i(t)$, $i=0,1$). This completes the study started in \cite{Do2020,FMS-RL3}, which only considered the case $\theta_0<\theta_1$. As an application we recover and generalize interpolation identities for grand and small Lebesgue spaces given by \cite{FFGKR}.
\end{abstract}
\maketitle

\section{Introduction}\label{introduction}
This paper continues the project started in \cite{Do2020,Do2021,FMS-RL1,FMS-RL2,FMS-RL3} where 
we have proved
%we have proved 
reiteration theorems for couples formed by arbitrary combinations of the spaces
$$\overline{X}_{\theta,\b,E}, 
\quad \overline{X}^{\mathcal R}_{\theta,\b,E,\a,F},\quad \overline{X}^{\mathcal L}_{\theta,\b,E,\a,F}$$
when $0\leq\theta\leq1$, $a$ and $\b$ are slowly varying functions and $E$, $F$ are rearrangement invariant (r.i.) Banach function  spaces or Lebesgue spaces $L_q$, $0<q\leq\infty$; see \S 4 for precise definitions.
In particular,  we have identified the spaces 
\begin{equation}\label{e2}
\Big(\overline{X}^{\mathcal R}_{\theta_0, \b_0,E_0,a_0,F_0}, \overline{X}^{\mathcal L}_{\theta_1,\b_1,E_1,a_1,F_1}\Big)_{\eta,\b,E} %= \overline{X}_{\theta,\B_\eta,E}, 
\:,\qquad 0\leq\eta\leq 1, 
\end{equation}
under the condition $\theta_0\neq\theta_1$.

Our goal in this paper is to identify \eqref{e2} in the special case $\theta_0=\theta_1$, provided that $a_0=a_1$, $F_0=F_1$ and for suitable $\b_0$ and $\b_1$;
%satisfying $\b_i(t)\sim\b_i(t^2)$, $i=0,1$
 that is the space
\begin{equation}\label{e3}
\Big(\overline{X}^{\mathcal R}_{\theta,\b_0,E_0,\a,F}, \overline{X}^{\mathcal L}_{\theta, \b_1,E_1,\a,F}\Big)_{\eta,\b,E}\:, \qquad 0\leq \eta\leq 1.
\end{equation}
The resulting spaces in \eqref{e3} are very different from those obtained in \eqref{e2}. While in the case $\theta_0\neq \theta_1$ the spaces belong to the classical scale $\overline{X}_{\eta,\b,E}$ for all $0<\eta<1$, now \eqref{e3} will always be an ${\mathcal R}$ or ${\mathcal L}$-space, depending on the values of $\eta$;  see Theorem \ref{maintheorem} for the precise statements. Moreover, in the cases $\eta=0$ or $\eta=1$ we will need intersections of ${\mathcal R}$ or ${\mathcal L}$-space with the extreme scales of $(\mathcal R,\mathcal L)$-spaces or $(\mathcal L,\mathcal R)$-spaces, respectively. These extremal constructions were introduced in \cite{FMS-RL1}.

\
 
As in \cite{Do2020,FMS-RL1,FMS-RL2,FMS-RL3}, the present work finds its motivation in applications to the interpolation of 
 \textit{grand} and \textit{small Lebesgue} spaces $L^{p),\alpha}$, $L^{(p,\alpha}$, $1<p<\infty$, $\alpha>0$ (see Definition \ref{def_gLp} below). These spaces are limiting interpolation spaces for the couples $(L_1,L_p)$ and $(L_p,L_\infty)$, respectively, and they can also be identified as ${\mathcal R}$ and ${\mathcal L}$-spaces  in the following way
$$
L^{p),\alpha}=(L_1,L_p)_{1,\ell^{-\frac{\alpha}{p}}(u),L_\infty}=(L_1,L_\infty)^{\mathcal R}_{1-\frac1p,\ell^{-\frac{\alpha}{p}}(t),L_\infty,1,L_p}$$
and $$L^{(p,\alpha}=(L_p,L_\infty)_{0,\ell^{\alpha(1-\frac{1}{p})-1}(u),L_1}=(L_1,L_\infty)^{\mathcal L}_{1-\frac1p,\ell^{\alpha(1-\frac{1}{p})-1}(t),L_1,1,L_p}
$$
where $\ell(t)=1+|\log(t)|$, $t\in(0,1)$. Then, from our results concerning \eqref{e3} we will obtain identities for
the interpolation spaces
\begin{equation*}\label{ec6}
\big(L^{p),\alpha},L^{(p,\beta}\big)_{\eta,\b,E}\:,\qquad 0\leq \eta\leq 1.
\end{equation*}
These recover (and extend) recent results of
Fiorenza et al. \cite[Theorem 6.2]{FFGKR}, 
which they obtained in the case $\alpha=\beta=1$, $\eta\in(0,1)$, $\b\equiv1$ and $E=L_q$. See Corollary \ref{cor57} for the precise statements.

\
The paper has a very technical profile. Some of these details are worth to mention. For example, an important feature of the ${\mathcal R}$ and ${\mathcal L}$ spaces is that they can be seen as extrapolation spaces, so \eqref{e3} can be regarded as an extreme reiteration result. This allows us to use a direct approach following the classical reiteration theorems (see \cite{Bennett-Sharpley,Bergh-Lofstrom,triebel}). 
Also, the setting $\theta_0=\theta_1$ contains some subtleties which do not appear in \eqref{e2} when $\theta_0\not=\theta_1$. One of them is the use of \textit{limiting} Hardy type inequalities (Lemma \ref{limit Hardy}). As explained in \cite{EOP}, such inequalities cannot hold for arbitrary slowly varying functions, such as $\b(t)=1+|\log(t)|$, $t>0$, which have the same behavior near $0$ and near $\infty$, but they are true working for example with broken logarithms 
\begin{equation}\label{ellA}
\ell^{(\alpha, \beta)}(t)=\left\{\begin{array}{ll}
\ell^\alpha(t),& 0<t\leq1 \\
\ell^\beta(t)
,& t>1
\end{array}\right.,\quad (\alpha, \beta)\in\R^2,
\end{equation}
provided $\alpha$ and $\beta$ have different signs (see \cite[Lemma 4.2]{EOP}). In this sense, we follow the ideas developed in \cite{FMS-3}, that is, we work  with slowly varying functions  satisfying  that $\b(t)\sim\b(t^2)$ and we define the associated functions $B_0$ and $B_\infty$ by 
$$\B_0(u)=\b(e^{1-\frac1u})\qquad\mbox{ and }\qquad \B_\infty(u)=\b(e^{\frac1u-1}), \qquad \mbox{for}\quad 0<u\leq1.$$ 
Under appropriate conditions in the \textit{extension indices} of $B_0$ and $B_\infty$ the limiting Hardy type inequalities in the whole line $(0,\infty)$ are true.

Another ingredient that we will use is a new identity between the ${\mathcal R}$-spaces and the $({\mathcal R,\mathcal L})$-spaces (or ${\mathcal L}$-spaces with  $({\mathcal L,\mathcal R})$-spaces, respectively); see Corollary \ref{cor45}. These are based on 
new equivalences for norm functionals involving suitable slowly varying functions;  see Lemma \ref{key1}. It must be said that the proof of Lemma \ref{key1} is very technical and needs the discretization of the norm of the spaces involved as well as weighted Hardy type inequalities for r.i. sequence spaces (see Lemma \ref{lema8}).

\

The organization of the paper is the following: In Section 2 we review basic concepts about rearrangement invariant spaces, extension indices of functions and slowly varying functions. In Section 3 we restrict to  slowly varying functions which satisfy that $\b(t)\sim\b(t^2)$, collect some essential lemmas (equivalence lemma, limiting estimates and limiting Hardy type inequalities) and prove the technical lemmas that we shall need for the identification of the spaces in \eqref{e3}. The description of the interpolation methods we shall work with, namely $\overline{X}_{\theta,\b,E}$, the ${\mathcal R}$, ${\mathcal L}$, $(\mathcal R,\mathcal L)$ and $(\mathcal L,\mathcal R)$ constructions  can be found in Section 4. Generalized Holmstedt type formula, the change of variables  and the main reiteration theorem appear in Section 5. Finally, Section 6 is devoted to applications to the interpolation of grand and small Lebesgue spaces.
 
\section{Preliminaries} \label{preliminaries}

We refer to the monographs \cite{Bennett-Sharpley,Bergh-Lofstrom,B-K,KPS,triebel} for the main definitions and properties concerning r.i. spaces and interpolation theory. Recall that a Banach function space $E$ (on  the semiaxis $\Omega=(0,\infty)$ with the Lebesgue measure) is called \textit{rearrangement invariant} (r.i.) if, for any two measurable functions $f$, $g$,
\begin{equation*}
g\in E \text{ and } f^*\leq g^*  \Longrightarrow f\in E \text{ and } \|f\|_E\leq \|g\|_E,
\end{equation*}
where $f^*$ and $g^*$ stand for the non-increasing rearrangements of $f$ and $g$. We  shall denote the r.i. space by $\textbf{e}$, when $\Omega=\Z$ with the counting measure.
Moreover, following  \cite{Bennett-Sharpley},  we always assume that every Banach function space $E$ enjoys the \textit{Fatou property}. Under this assumption,  every r.i. space
$E$ is an exact interpolation space with respect to the Banach couple $(L_{1}, L_{\infty})$, that is
$$E=(L_1,L_\infty)_D^K$$
for some suitable choice of \textit{parameter} $D$.

Let us recall more explicitly what this notation means. For any compatible (quasi-) Banach couple $\overline{X}=(X_0,X_1)$, the \textit{Peetre $K$-functional} $K(t,f;X_0,X_1)\equiv K(t,f)$ is defined for $f\in X_0+X_1$ and $t>0$ by
 \begin{align*}
K(t,f)=\inf \Big \{\|f_0\|_{X_0}+t\|f_1\|_{X_1}:\ f=f_0+f_1,\ f_i\in X_i , \; i=0,1
\Big \}.
\end{align*}
A Banach lattice $D$ of measurable functions on $(0,\infty)$ is called a parameter if it contains the function $\min(1,t)$. For each such $D$, the real interpolation space $\overline{X}_D^K$ consists of all $f\in X_0+X_1$, for which
$$\|f\|_{\overline{X}_D^K}=\|K(t,f)\|_D<\infty.$$

We shall work with three different measures on $(0,\infty)$: the usual Lebesgue measure $dt$, the homogeneous measure $dt/t$ and the measure $dt/t\ell(t)$ where $\ell(t):=(1+|\log t|)$, $t>0$. 
As in \cite{FMS-3}, we use letters with a tilde for spaces with the measure $dt/t$ and with a hat for the measure $dt/t\ell(t)$. For example, the spaces $\widetilde{L}_{1}$  and $\widehat{L}_1$ are defined by the norms
$$  \|f\|_{\widetilde{L}_{1}} = \int_{0}^{\infty}  |f(t)|  \tfrac{dt}{t} \qquad \mbox{and} \qquad 
\|f\|_{\widehat{L}_1} = \int_{0}^{\infty}  |f(t)|  \tfrac{dt}{t\ell(t)},$$
respectively, while $\widetilde{L}_{\infty} $  and $\widehat{L}_\infty$ coincide with $ L_{\infty}$. Similarly, each r.i. space $E$ on $((0,\infty),dt)$ generates two more versions of $E$ on the other measure spaces. Namely, if $D$ is a parameter such that $E=(L_1,L_\infty)_D^K$, then we set
$$\widetilde{E}=(\widetilde{L}_1,L_\infty)_D^K \mand \widehat{E}=(\widehat{L}_1,L_\infty)_D^K.$$

It should be pointed out that the definition of the spaces $\widetilde{E}$ and $\widehat{E}$ do not depend on any particular choice of parameter $D$. In fact, there are simple formulae which directly connect the norms of the spaces $E$, $\widetilde{E}$  and $\widehat{E}$, without any reference to the parameter $D$. 
For all measurable function $f:(0,\infty)\longrightarrow (0,\infty)$ we have
$$\|f\|_{\widetilde{E}(0,1)}=\|f(e^{-u})\|_{E} \quad \text{ and }
\quad \|f\|_{\widetilde{E}(1,\infty)}=\|f(e^u)\|_{E},$$
while
$$\| f \|_{\widehat{E}(0,1)} =  
\| f ( e^{1 - e^u}) \|_{E} \quad \text{and} \quad \|f\|_{\widehat{E} (1,\infty)} 
=  \| f ( e^{ e^u -1 }) \|_{E}.$$

%For the rest of the paper we shall denote by $\overline{f}$ the measurable function defined by $$\overline{f}(t)=f(1/t),\quad t>0.$$ Moreover, given $t>0$, it is sometimes convenient to divide the interval $(0,\infty)$ into two subintervals $(0,t)$ and $(t,\infty)$, considering the spaces $\widetilde{E}(0,t)$ and $\widetilde{E}(t,\infty)$ separately. Observe that $f\in \widetilde{E}(t,\infty)$ if and only if $\overline{f}\in \widetilde{E}(0,1/t)$ and
%\begin{equation}\label{fbar}
%\|f\|_{\widetilde{E}(t,\infty)}=\|\overline{f}\|_{\widetilde{E}(0,1/t)}. 
%\end{equation}

Throughout the paper, given two (quasi-) Banach spaces $X$ and $Y$, we will write $X=Y$ if $X\hookrightarrow Y$ and $Y\hookrightarrow X$, where the latter means that $Y\subset X$ and the natural embedding is continuous.
Similarly, $f\sim g$ means that  $f\lesssim g$ and $g\lesssim f$, where $f\lesssim g$ is the abbreviation of $f(t) \leq C g(t)$, $t >0$, for some positive constant $C$ independent of $f$ and $g$. Moreover, we do not distinguish the notions ``increasing" and ``non-decreasing" as well as ``decreasing" and ``non-increasing" and we say a function $f$ is almost increasing (decreasing) if it is equivalent to an increasing (decreasing) function.

\subsection{Extension indices} Given a non-negative everywhere finite function $\varphi$ on $(0,a)$, $0<a\leq \infty$, we consider its \textit{dilation function} 
$$m_\varphi(t)=\sup_{0<s<\min\{a,\frac{a}{t}\}}\frac{\varphi(ts)}{\varphi(s)},\qquad 0<t<\infty.$$ 
If $m_\varphi(t)$ is finite everywhere, then there exist the \textit{lower} and \textit{upper extension indices} of $\varphi$ given by
$$\pi_\varphi=\lim_{t\rightarrow 0}\frac{\ln m_\varphi(t)}{\ln
t},\qquad \rho_\varphi=\lim_{t\rightarrow \infty}\frac{\ln
m_\varphi(t)}{\ln t}.$$ 
In general, $-\infty<\pi_\varphi\leq\rho_\varphi<\infty$, but $0\leq \pi_\varphi\leq\rho_\varphi<\infty$ if the function $\varphi$ is increasing and $0\leq \pi_\varphi\leq\rho_\varphi\leq1$ if  $\varphi$ is quasi-concave.
Moreover, if $0<\pi_\varphi\leq\rho_\varphi<\infty$, then 
$\varphi$ is equivalent to the increasing function $\int_0^t\varphi(s)\frac{ds}{s}$ while if $-\infty<\pi_\varphi\leq\rho_\varphi<0$ then the function $\varphi$ is equivalent to the decreasing function $\int_t^\infty\varphi(s)\frac{ds}{s}$ (see \cite{KPS} pg. 57).
As an example, $\varphi(t)=t^\alpha \ell(t)^\beta$, $\alpha,\beta\in\mathbb{R}$, has $\pi_\varphi=\rho_\varphi=\alpha$.

The following properties of extension indices will be useful in the sequel and can be easily proved:
\begin{itemize}
\item[(i)] Let $\sigma\in\R$ and $\varphi(t)=t^{\sigma}\psi(t)$, then $\pi_{\varphi}=\sigma+\pi_{\psi}$ and $\rho_{\varphi}=\sigma+\rho_{\psi}$.
\item[(ii)] If $\varphi(t)=\alpha(t)\psi(t)$ then $\pi_{\varphi}\geq \pi_{\alpha}+\pi_{\psi}$ and $\rho_{\varphi}\leq \rho_{\alpha}+\rho_{\psi}$.
\item[(iii)] If $\varphi(t)=\psi(t^{-1})$ then $\pi_{\varphi}=-\rho_{\psi}$ and $\rho_{\varphi}=-\pi_{\psi}$. 
\item[(iv)] If $\varphi(t)=1/\psi(t)$ then $\pi_{\varphi}=-\rho_{\psi}$ and $\rho_{\varphi}=-\pi_{\psi}$.
\item[(v)]If $\varphi(t)=\theta(\psi(t))$ and $\theta$ is almost increasing then $\pi_{\varphi}\geq \pi_{\theta}\pi_{\psi}$ and $\rho_{\varphi}\leq \rho_{\theta}\rho_{\psi}$.
\item[(vi)] If $\varphi\sim \psi$ then $\pi_{\varphi}=\pi_{\psi}$ and $\rho_{\varphi}=\rho_{\psi}$.
\end{itemize}
From (i) it follows that the ratio $\varphi(t)/t^\sigma$ is almost increasing for any $\sigma<\pi_\varphi$ and almost decreasing for any $\sigma>\rho_\varphi$.

Every r.i. space $E$ has its fundamental function
$\varphi_E(\lambda)=\|\chi_{(0,\lambda)}\|_E$, which is continuous
and quasi-concave.  Moreover, the space $E$ always admits an
equivalent renorming such that $\varphi_E$ becomes
concave and the derivative $\varphi'_E$ exists a.e. and is
decreasing. In particular, $0\leq \pi_{\varphi_E}\leq \rho_{\varphi_E}\leq 1$.

\subsection{Slowly varying functions} 
In this subsection we recall the definition and basic properties of \textit{slowly varying functions}.  See  \cite{BGT,GOT,Matu}.

\begin{defn}\label{def1}
A positive Lebesgue measurable function $\b$, $0\not\equiv\b\not\equiv\infty$,
is said to be \textit{slowly varying} on $(0,\infty)$ (notation $\b\in SV$) if, for each $\varepsilon>0$, the function $t \leadsto t^\varepsilon\b(t)$  is  almost increasing on $(0,\infty)$
and $ t \leadsto t^{-\varepsilon}\b(t)$ is  almost decreasing on $(0,\infty)$.
\end{defn}

The class $SV(0, 1)$ can be defined similarly, changing the interval $(0,\infty)$ by $(0,1)$. Examples of $SV$-functions include powers of logarithms,
$\ell^\alpha(t)=(1+|\log t|)^\alpha$, rei\-terated logarithms $(\ell\circ\ldots\circ\ell)^\alpha (t),\ \alpha\in\R,\ t>0$, ``broken" logarithmic functions $\ell^{\alpha,\beta}(t)$ (see \eqref{ellA}) 
 and also the family of functions as $t \leadsto \exp(|\log t|^\alpha), \ \alpha \in (0,1), \; t >0$. 

Some basic properties of slowly varying functions are summarized in the following lemmas.

\begin{lem}\label{lem0}
Let $\b, \b_1, \b_2\in SV$ and let $\mu$ be a non-negative measurable function on $(0,\infty)$. 
\begin{itemize}
\item[(i)]  Then $\b_1\b_2\in SV$, $\b(1/t)\in SV$ and $\b^r\in SV$ for all $r\in\R$.
\item[(ii)] If for some $\delta>0$, the functions $\mu$ and $t/\mu^\delta(t)$ are almost increasing  on $(0,\infty)$, then $\b\circ\mu\in SV$.
\item[(iii)] If $\epsilon,s>0$ then there are positive constants $c_\epsilon$ and $C_\epsilon$ such that
$$c_\epsilon\min\{s^{-\epsilon},s^\epsilon\}\b(t)\leq \b(st)\leq C_\epsilon \max\{s^\epsilon,s^{-\epsilon}\}\b(t)\quad \mbox{for every}\  t>0.$$
In particular, $\pi_\b=\rho_\b=0$.
\item[(iv)] $\b\circ f\sim \b\circ g$ if $f$ and $g$ are positive finite equivalent functions on $(0,\infty)$.
\item[(v)] If $\b_1$ is almost monotone, then $\b\circ\b_1\in SV$.
\end{itemize}
\end{lem}

\vspace{2mm}
\begin{lem} \label{lem1}
Let $E$ be an r.i. space on $(0,\infty)$ and $\b\in SV$.
\begin{itemize}
\item[(i)] If $\alpha>0$, then, for all $t>0$,
$$\|s^\alpha\b(s)\|_{\widetilde{E}(0,t)}\sim t^\alpha\b(t)\quad \mbox{ and } \quad
\|s^{-\alpha}\b(s)\|_{\widetilde{E}(t,\infty)}\sim t^{-\alpha}\b(t).$$
\item[(ii)] If $\alpha\in\R$, then, for all $t>0$,
$$\|s^\alpha\b(s)\|_{\widetilde{E}(t,2t)}\sim t^\alpha\b(t).$$
\item[(iii)] The following functions belong to $SV$
$$t \leadsto\|\b\|_{\widetilde{E}(0,t)}
\quad \text{and}\quad
t \leadsto\|\b\|_{\widetilde{E}(t,\infty)}, \quad t >0.$$
\item[(iv)] For all $t >0$,
$$\b(t) \lesssim \| \b \|_{\widetilde{E}(0,t)} \quad \text{and} \quad
\b(t) \lesssim \| \b \|_{\widetilde{E}(t,\infty)}.$$
\end{itemize}
\end{lem}

We refer to \cite{GOT,FMS-1,FMS-4} for the proof of Lemma \ref{lem0} and \ref{lem1}, respectively.
\begin{rem}\label{rem23}
\textup{ The property (iii) of Lemma \ref{lem0} implies that if $\b\in SV$ is such that $\b(t_0)=0$ ($\b(t_0)=\infty$) for some $t_0>0$, then $\b\equiv0$ ($\b\equiv\infty$). Thus, by Lemma \ref{lem1} (ii), if $\|\b\|_{\widetilde{E}(0,1)}<\infty$ then
$\|\b\|_{\widetilde{E}(0,t)}<\infty$ for all $t>0$, and if $\|\b\|_{\widetilde{E}(1,\infty)}<\infty$ then
$\|\b\|_{\widetilde{E}(t,\infty)}<\infty$ for all $t>0$.}
\end{rem}

\section{Especial slowly varying functions.}

In this paper we are concerned with those slowly varying functions $\b$ such that $\b(t^2)\sim\b(t)$. All previous examples, except $\b(t)=exp(|\log t|^\alpha)$, $\alpha\in(0,1)$, satisfy this condition. For every $\b\in SV$ with this property, we define new functions $\B_0$ and $\B_\infty$ by 
\begin{equation}\label{fb}
\B_0(u)=\b(e^{1-\frac1u})\qquad\mbox{ and }\qquad \B_\infty(u)=\b(e^{\frac1u-1}), \qquad \mbox{for}\quad 0<u\leq1.
\end{equation}

For example, if $\b(t)=\ell^{(\alpha,\beta)}(t)$ then $B_0(u)=1/u^\alpha$ and $B_\infty(u)=1/u^\beta$, $u\in(0,1]$. 
The condition $\b(t^2)\sim\b(t)$ implies that $\B_0$ and $\B_\infty$ satisfy the $\Delta_2$ condition, that is $\B_0(u)\sim \B_0(2u)$ and 
$\B_{\infty}(u)\sim \B_{\infty}(2u)$. 
%so that $m_{\B_0}(t)<\infty$ and $m_{\B_\infty}(t)<\infty$ for all $t>0$. 
Hence the extension indices of $\B_0$ and $\B_\infty$ exist and are both finite.

In what follows we need to consider different functions with the above properties, so we shall denote these functions by $a$, $\b$, $\phi$, etc, and their associated functions  by $A_0$, $A_\infty$, $B_0$, $B_\infty$, $\Phi_0$, $\Phi_\infty$, etc (with capital letters).

Next, we collect some useful lemmas  about this class of slowly varying functions from \cite{FMS-3}.  The first lemma gives the counterparts of Lemma \ref{lem1} i) in the limiting case $\alpha=0$. The third one contains limiting Hardy-type inequalities.

\begin{lem}\cite[Lemma 3.2]{FMS-3} \label{lema2.3}
Let $E$ be an r.i. space and let $\b\in SV$ such that $\b(t^{2}) \sim \b(t)$, with associated functions $\B_0$ and $\B_\infty$ defined in (\ref{fb}).
\begin{itemize}
\item[i)] If $\rho_{\B_{\infty}} < \pi_{\varphi_E} \leq \rho_{\varphi_E} < \pi_{\B_{0}}$, then 
$$ \|  \b \|_{\widetilde{E}(0,t)} \sim \b(t) \varphi_{E}(\ell(t)),\qquad t>0.$$
\item[ii)] If $\rho_{\B_{0}} < \pi_{\varphi_E} \leq \rho_{\varphi_E} < \pi_{\B_{\infty}}$, then 
the equivalence is
$$\| \b \|_{\widetilde{E}(t,\infty)} \sim \b(t)  \varphi_{E} (\ell(t)),\qquad t>0.$$
\end{itemize}
\end{lem}

%\begin{lem} \label{lema2.4}
%Let $E$ be an r.i. space, $\b\in SV$ such that $\b(t^{2}) \sim \b(t)$ and let $\B_0$ and $\B_\infty$ be its associated functions defined by (\ref{fb}).
%
%\vspace{3mm}
%\noindent\textup{(i)} If $ \rho_{\B_{\infty}} < 0 < \pi_{\B_{0}}$, then 
%$$ \|  \b \|_{\widehat{L}_1(0,t)} \lesssim \b(t),\qquad t>0.$$
%\textup{(ii)} If $\rho_{\B_{0}} < 0 < \pi_{\B_{\infty}}$, then 
%
%$$\| \b \|_{\widehat{L}_1(t,\infty)} \lesssim \b(t),\qquad t>0.$$
%\end{lem}
%
%\begin{proof}
%For $0 < t \leq1$ we use  the expression of $\b$ on $(0,1)$ and the change of variables, $\ell(s)=\frac1u$, we obtain
%$$\int_0^t\b(s)\frac{ds}{s\ell(s)}=\int_0^t\B_0\Big(\frac1{\ell(s)}\Big)\frac{ds}{s\ell(s)}=\int_0^{\frac1{\ell(t)}}\B_0(u)\frac{du}u.$$
%From hypothesis $0 < \pi_{\B_{0}}$, then using Corollary 1 of  \cite[p. 57]{KPS} we have that
%$$\int_0^t\b(s)\frac{ds}{s\ell(s)}\sim \B_0\Big(\frac1{\ell(t)}\Big)=\b(t).$$
%Let now $1<t<\infty$. Using the previous case and the same change of variables
%$$\int_0^t\b(s)\frac{ds}{s\ell(s)}\sim\b(1)+\int_1^t\B_\infty\Big(\frac1{\ell(s)}\Big)\frac{ds}{s\ell(s)}=
%\b(1)+\int_{\frac{1}{\ell(t)}}^\infty\B_\infty(u)\frac{du}{u}.$$
%The function $B_\infty$ has the property $\rho_{\B_{\infty}} < 0$ hence (see \cite[p. 57]{KPS}) 
%$$\int_0^t\b(s)\frac{ds}{s\ell(s)}\sim\b(1)+B_\infty\Big( \frac{1}{\ell(t)}\Big)\leq\b(t).$$
%
%Similarly it can be proved the estimate (ii).
%\end{proof}

\begin{lem}\cite[Lemma 3.5]{FMS-3} \label{embeddings II}
Let $E$ be an r.i. space and let $\b\in SV$ such that $\b(t^2)\sim\b(t)$, with associated functions $\B_0$ and $\B_\infty$ defined in (\ref{fb}).
\begin{itemize}
\item[i)] If $0 < \pi_{B_{0}}$ and $\varphi$ is an almost decreasing function, then for all $t>0$
$$ \| \b(s) \varphi(s)    \|_{\widetilde{E}(0,t)}    \lesssim   
\int_{0}^{t}  \b(s) \varphi(s) \varphi_{E}( \ell (s))  \frac{ds}{s\ell (s)}.$$ 

\item[ii)] If $0<\pi_{B_{\infty}}$ and $\varphi$ is an almost increasing function, then for all $t>0$
$$  \|  \b(s) \varphi(s)   \|_{\widetilde{E}(t,\infty)}    \lesssim   
\int_{t}^{\infty}  \b(s) \varphi(s)  \varphi_{E}( \ell (s))  \frac{ds}{s\ell (s)}.$$
\end{itemize}
\end{lem}

\begin{lem}\cite[Lemma 3.6]{FMS-3}\label{limit Hardy}
Let $E$ be an r.i. space and let $\b\in SV$ such that $\b(t^2)\sim\b(t)$, with associated functions $\B_0$ and $\B_\infty$ defined in (\ref{fb}).
\begin{itemize}
\item[i)] If $\rho_{\B_{0}} < 0<\pi_{\B_{\infty}}$, then 
$$
\Big \| \b(t)  \int_{0}^{t} f(s)\, ds \Big \|_{\widehat{E}}  
 \lesssim \big \|t \b(t) f(t)  \ell(t) \big \|_{\widehat{E}}$$
for any non-negative measurable function $f$ on $(0,\infty)$.

\item[ii)] If %, on the contrary, the indices of $\B_0$ and $\B_\infty$ satisfy 
$\rho_{\B_{\infty}} <0<\pi_{\B_{0}} $, then 
$$  \Big \| \b(t)  \int_{t}^{\infty}  f(s) \,  ds \Big \|_{\widehat{E}}   
  \lesssim \big \|t \b(t)  f(t)  \ell(t) \big \|_{\widehat{E}}$$
for any non-negative measurable function $f$ on $(0,\infty)$.
\end{itemize}
\end{lem}

\subsection{Technical lemmas}
This subsection contains the most technical lemmas that we shall need in the proof of the main result of this paper.

 The first one is a weighted Hardy type inequality, which is well known for Lebesgue sequence spaces  $e=\ell_q$, $0<q\leq\infty$ (see \cite{GP,GHS}).
 In our case, we shall need it for any r.i. sequence space $\textbf{e}$, that is, spaces of the form $\textbf{e}=(\ell_1,\ell_\infty)^K_D$ for some suitable choice of the parameter $D$.
\begin{lem}\label{lema8}
Let $(\sigma_k)_{k\in\Z}$ be a non-negative sequence and let $\textbf{e}$ be an r.i. sequence space.
\begin{itemize}
\item[i)] If $\sup_{k\in\Z}\frac{\sigma_{k+1}}{\sigma_k}<1$, then 
\begin{equation}\label{e1lema8}
\Big\|\Big(\sigma_k\sum_{m\leq k}x_m\Big)_{k\in\Z}\Big\|_{\textbf{e}}\sim \|(\sigma_k x_k)_{k\in\Z}\|_{\textbf{e}}
\end{equation}
for any non-negative sequence $(x_k)_{k\in\Z}$.
\item[ii)] If $\inf_{k\in\Z}\frac{\sigma_{k+1}}{\sigma_k}>1$, then 
\begin{equation}\label{e2lema8}
\Big\|\Big(\sigma_k\sum_{m\geq k}x_m\Big)_{k\in\Z}\Big\|_{\textbf{e}}\sim \|(\sigma_k x_k)_{k\in\Z}\|_{\textbf{e}}
\end{equation}
for any non-negative sequence $(x_k)_{k\in\Z}$.
\end{itemize}
\end{lem}
\begin{proof}
Let $\sigma:=\sup_{k\in\Z}\frac{\sigma_{k+1}}{\sigma_k}<1$. The estimate ``$\gtrsim$" in \eqref{e1lema8} is clearly true, so only the estimate ``$\lesssim$'' has to be proved. 
We consider the operator $R:\textbf{e}\rightarrow\textbf{e}$ defined by
$$R((y_k)_{k\in\Z})=\Big(\sigma_k\sum_{m\leq k}\frac{y_m}{\sigma_m}\Big)_{k\in\Z}.$$
We claim that $R$ is bounded in $\ell_1$ and in $\ell_\infty$. Actually, let $(y_k)_{k\in\Z}\in\ell_1$, then
\begin{align*}
\|R((y_k)_{k\in\Z})\|_{\ell_1}&=\sum_{k\in\Z}\Big|\sigma_k\sum_{m\leq k}\frac{y_m}{\sigma_m}\Big|\leq \sum_{k\in\Z}\sigma_k\sum_{m\leq k}\frac{|y_m|}{\sigma_m}
=\sum_{m\in\Z}\frac{|y_m|}{\sigma_m}\sum_{k\geq m}\sigma_k\\
&\leq \sum_{m\in\Z}|y_m|(1+\sigma+\sigma^{2}+\ldots)=\frac{1}{1-\sigma}\|(y_m)_{m\in\Z}\|_{\ell_1}.
\end{align*}
Take now $(y_k)_{k\in\Z}\in\ell_\infty$, then
\begin{align*}
\|R((y_k)_{k\in\Z})\|_{\ell_\infty}&=\sup_{k\in\Z}\Big|\sigma_k\sum_{m\leq k}\frac{y_m}{\sigma_m}\Big|\leq \|(y_m)_{m\in\Z}\|_{\ell_\infty}\sup_{k\in\Z}\Big\{\sigma_k\sum_{m\leq k}\frac{1}{\sigma_m}\Big\}\\
&\leq \|(y_m)_{m\in\Z}\|_{\ell_\infty}(1+\sigma+\sigma^{2}+\ldots)=\frac{1}{1-\sigma}\|(y_m)_{m\in\Z}\|_{\ell_\infty}.
\end{align*}
Since $\textbf{e}$ is an interpolation space for the couple $(\ell_1,\ell_\infty)$ we have that the operator $R$ is bounded in $\textbf{e}$. To complete the proof of i) it suffices to take $y_k=\sigma_kx_k$, $k\in\Z$. 
The proof of ii) can be done in a similar way.
\end{proof}

In next lemmas we shall consider the following sequence 
\begin{equation}\label{lambda}
\lambda_{k} = \begin{cases}
 e^{1-e^{-k}}& \text{ if }  -k \in \N ,\\
            e^{e^{k}-1} & \text{ if }  k \in \N \cup \{0\}.
                \end{cases}
\end{equation}
Observe that the special slowly varying functions we are working with, $\b\in SV$ such that $\b(t^2)\sim\b(t)$,
satisfy that
\begin{equation}\label{equiv}
\b(\lambda_{k-1})\sim\b(t)\sim \b(\lambda_k)\quad \mbox{for}\ t\in[\lambda_{k-1},\lambda_k],
\end{equation}
with the equivalent constants  being independent of $k\in\Z$. 

\begin{lem}\label{lema29}
Let $(\lambda_k)_{k\in\Z}$ be the sequence defined by \eqref{lambda} and let $\b\in SV$ such that $\b(t^2)\sim\b(t)$, with associated functions $\B_0$ and $\B_\infty$ defined in \eqref{fb}.

\begin{itemize}
\item[i)] If $\rho_{B_0}<0<\pi_{B_\infty}$, then there exist a measurable function $\Phi$ equivalent to $\b$ such that $\sup_{k\in\Z}\frac{\Phi(\lambda_{k+1})}{\Phi(\lambda_k)}<1$.
\item[ii)] If $\rho_{B_\infty}<0<\pi_{B_0}$, then there exist a measurable function $\Psi$ equivalent to $\b$ such that $\inf_{k\in\Z}\frac{\Psi(\lambda_{k+1})}{\Psi(\lambda_k)}>1$.
\end{itemize}
\end{lem}
\begin{proof}
An straightforward computation shows that
\begin{equation*}
\b(\lambda_{k}) = \begin{cases}
 B_0(e^k)& \text{ if }  -k \in \N ,\\
            B_\infty(e^{-k}) & \text{ if }  k \in \N \cup \{0\}.
                \end{cases}
\end{equation*}
Assume now that $\rho_{B_0}<0<\pi_{B_\infty}$. Then, there exist $\alpha$, $\beta\in\R$ such that $\rho_{\B_{0}}<\alpha< 0<\beta <\pi_{\B_{\infty}}$. By properties of extension indices (see, e.g., \cite[Section II.1.2]{KPS}), the function $t^{-\alpha}B_0(t)$ is almost decreasing and $t^{-\beta}B_\infty(t)$ is almost increasing. This implies the existence of two functions $\Phi_0\sim B_0$ and $\Phi_\infty\sim B_\infty$ such that $t^{-\alpha}\Phi_0(t)$ is strictly decreasing while  $t^{-\beta}\Phi_\infty(t)$ is strictly increasing. So, it follows that the function 
\begin{equation*}
\phi(\lambda_{k}) = \begin{cases}
 \Phi_0(e^k)& \text{ if }  -k \in \N ,\\
            \Phi_\infty(e^{-k}) & \text{ if }  k \in \N \cup \{0\},
                \end{cases}
\end{equation*}	
satisfies that $\sup_{k\in\Z}\frac{\phi(\lambda_{k+1})}{\phi(\lambda_k)}< max\{e^\alpha,e^{-\beta}\}<1$. The argument to prove ii) is analogous.
\end{proof}

In the next two lemmas we shall need to work with the discretization of $\widehat{E}$ that we denote by $d\widehat{E}$. Let us explain more explicitly the definition of $d\widehat{E}$. Consider the intervals
\begin{equation*}
I_{k} = \begin{cases}
        (\lambda_{k - 1}, \lambda_{k}] \quad \text{ if }   -k  \in \N, \\
        (\lambda_{- 1}, \lambda_{0}) \quad \ \text{ if }  k=0, \\
        [\lambda_{k - 1}, \lambda_{k}) \quad  \text{ if }   k  \in \N
        \end{cases}
\end{equation*}
where $(\lambda_k)_{k\in\Z}$ is the sequence defined by \eqref{lambda}. Hence, $\bigcup_{k\in\Z} I_k$ is a decomposition of $(0,\infty)$ into disjoint intervals $I_k$  such that
$$\int_{I_k}\frac{dt}{t\ell(t)}=1, \quad \forall k\in\Z.$$
It is easy to prove that the operator
$$Tf=\Big(\int_{I_k}f(t)\frac{dt}{t\ell(t)}\Big)_{k\in\Z}$$
is bounded from $\widehat{L}_1$ into $\ell_1(\Z)$ and from $L_\infty$ into $\ell_{\infty}(\Z)$. By interpolation,  $T$ is bounded from $\widehat{E}=(\widehat{L}_1,L_\infty)_D^K$ into $d\widehat{E}:=(\ell_1(\Z),\ell_\infty(\Z))_D^K$ for the suitable choice of parameter $D$. Similarly, the operator
$$S((x_k)_{k\in\Z})=\sum_{k\in\Z}x_k\chi_{I_k}$$
is bounded from $\ell_1(\Z)$ into $\widehat{L}_1$ and from $\ell_{\infty}(\Z)$ into $L_\infty$. Then, using interpolation $S$ is bounded from $d\widehat{E}=(\ell_1(\Z),\ell_\infty(\Z))_D^K$ into $\widehat{E}=(\widehat{L}_1,L_\infty)_D^K$. Moreover, $TSx=x$ for all $x\in d\widehat{E}$ and $STf=f$ for all $f\in \widehat{E}^c$, where $\widehat{E}^c$ denotes the collection of all functions on $\widehat{E}$ constant on each interval $I_k$.
Hence we can identify  
 $\widehat{E}^c$ with the space $d\widehat{E}$  of all sequences $x=(x_k)_{k\in\Z}$  such that 
$$\|x\|_{d\widehat{E}}=\Big\|\sum_{k\in\Z}x_k\chi_{I_k}\Big\|_{\widehat{E}}<\infty.$$
The space $d\widehat{E}$ is an r.i. sequence space, that we shall called the \textit{discretization} of $\widehat{E}$.

\begin{lem}\label{key1}
Let $E$, $F$, $G$ be r.i. spaces and let  $a$, $\b\in SV$ such that $a(t^2)\sim a(t)$, $\b(t^2)\sim\b(t)$, with  associated functions $A_0$, $A_\infty$, $\B_0$ and $\B_\infty$ defined in (\ref{fb}).
\begin{itemize}
\item[i)] If $\rho_{A_{\infty}} < \pi_{\varphi_F} \leq \rho_{\varphi_F} < \pi_{A_{0}}$ and $\rho_{B_0}<0<\pi_{B_\infty}$, then
$$\big \|\b(u) \|f\|_{ \widetilde{G}(0,u)} \big\|_{\widehat{E}}\sim\Big \|\frac{\b(u)}{\|a\|_{\widetilde{F}(0,u)}} \big\| a(t) \|f\|_{ \widetilde{G}(t,u)}\big\|_{ \widetilde{F}(0,u)} \Big\|_{\widehat{E}}$$
for any measurable function $f$ on $(0,\infty)$.
\item[ii)] If $\rho_{A_0} < \pi_{\varphi_F} \leq \rho_{\varphi_F} < \pi_{A_\infty}$ and $\rho_{B_\infty}<0<\pi_{B_0}$, then
$$\big \|\b(u) \|f\|_{ \widetilde{G}(u,\infty)} \big\|_{\widehat{E}}\sim\Big \|\frac{\b(u)}{\|a\|_{\widetilde{F}(u,\infty)}} \big\| a(t) \|f\|_{ \widetilde{G}(u,t)}\big\|_{ \widetilde{F}(u,\infty)} \Big\|_{\widehat{E}}$$
for any  measurable function $f$ on $(0,\infty)$.
\end{itemize}
\end{lem}

\begin{proof}
We observe that the inequality ``$\gtrsim$'' always holds, so in both cases only the inequality ``$\lesssim$'' has to be proved.
Let us begin with the proof of i). Using \eqref{equiv} and the discretization of the norm in $\widehat{E}$ we derive
\begin{align}\label{e34}
\big \|\b(u)  \|f\|_{ \widetilde{G}(0,u)} \big\|_{\widehat{E}}&=
\Big \|\sum_{k\in\Z}\b(u)  \|f\|_{ \widetilde{G}(0,u)}\chi_{I_k}(u) \Big\|_{\widehat{E}}\\
&\lesssim\Big \|\sum_{k\in\Z}\b(\lambda_k) \|f\|_{ \widetilde{G}(0,\lambda_k)}\chi_{I_k}(u) \Big\|_{\widehat{E}}\nonumber\\
&\leq\Big \|\sum_{k\in\Z}\b(\lambda_k)\Big(\sum_{j\leq k}\|f\|_{ \widetilde{G}(\lambda_{j-1},\lambda_j)}\Big)\chi_{I_k}(u) \Big\|_{\widehat{E}}\nonumber\\
&= \Big \|\Big(\b(\lambda_k) \sum_{j\leq k}\|f\|_{ \widetilde{G}(\lambda_{j-1},\lambda_j)}\Big)_{k\in\Z}\Big\|_{d\widehat{E}}.\nonumber
\end{align}
Now by Lemma \ref{lema29} i) ($\rho_{B_0}<0<\pi_{B_\infty}$), there exists an equivalent function of $\b$, that we denote in the same way, such that $\sup_{k\in\Z}\frac{\b(\lambda_{k+1})}{\b(\lambda_k)}<1$. Then, Lemma \ref{lema8} ii) implies that
$$
\Big \|\Big(\b(\lambda_k) \sum_{j\leq k}\|f\|_{ \widetilde{G}(\lambda_{j-1},\lambda_j)}\Big)_{k\in\Z}\Big\|_{d\widehat{E}}\sim
\Big \|\Big(\b(\lambda_k) \|f\|_{ \widetilde{G}(\lambda_{k-1},\lambda_k)}\Big)_{k\in\Z}\Big\|_{d\widehat{E}}.$$
Therefore, using again the discretization of the norm in $\widehat{E}$, we obtain
\begin{align*}
\Big \|\b(u)  \|f\|_{ \widetilde{G}(0,u)} \Big\|_{\widehat{E}}&\lesssim
\Big\|\Big(\b(\lambda_k)\|f\|_{ \widetilde{G}(\lambda_{k-1},\lambda_k)}\Big)_{k\in\Z} \Big\|_{d\widehat{E}}\\
&=\Big\|\sum_{k\in\Z}\b(\lambda_k)\|f\|_{ \widetilde{G}(\lambda_{k-1},\lambda_k)}\chi_{I_{k+1}}(u) \Big\|_{\widehat{E}}\\
&\sim\Big\|\sum_{k\in\Z}\frac{\b(u)}{\|a\|_{\widetilde{F}(0,u)}}\|a\|_{\widetilde{F}(0,u)}\|f\|_{ \widetilde{G}(\lambda_{k-1},\lambda_k)}\chi_{I_{k+1}}(u) \Big\|_{\widehat{E}}.
\end{align*}
Additionally, Lemma \ref{lema2.3} i) gives that $\|a\|_{\widetilde{F}(0,u)}\sim \a(u)\varphi_{F}(\ell(u))$. Hence, $\|a\|_{\widetilde{F}(0,u)}$ is a slowly varying function such that $\|a\|_{\widetilde{F}(0,u^2)}\sim \|a\|_{\widetilde{F}(0,u)}$. Thus, for any $u\in I_{k+1}$ we have that $\|a\|_{\widetilde{F}(0,u)}\sim\|a\|_{\widetilde{F}(0,\lambda_k)}\sim\|a\|_{\widetilde{F}(0,\lambda_{k-1})}$, $k\in\Z$. Moreover,
\begin{align*}
\|a\|_{\widetilde{F}(0,\lambda_{k-1})}\|f\|_{ \widetilde{G}(\lambda_{k-1},\lambda_k)}&=
\Big\|a(t)\|f\|_{ \widetilde{G}(\lambda_{k-1},\lambda_k)}\Big\|_{\widetilde{F}(0,\lambda_{k-1})}\lesssim\Big\|a(t)\|f\|_{ \widetilde{G}(t,\lambda_k)}\Big\|_{\widetilde{F}(0,\lambda_{k-1})}\\&\leq\Big\|a(t)\|f\|_{ \widetilde{G}(t,\lambda_k)}\Big\|_{\widetilde{F}(0,\lambda_{k})}.
\end{align*}
Summing up,
\begin{align*}
\Big \|\b(u)  \|f\|_{ \widetilde{G}(0,u)} \Big\|_{\widehat{E}}&\lesssim
\bigg\|\sum_{k\in\Z}\frac{\b(u)}{\|a\|_{\widetilde{F}(0,u)}}\Big\|a(t)\|f\|_{ \widetilde{G}(t,\lambda_{k})}\Big\|_{\widetilde{F}(0,\lambda_{k})}\chi_{I_{k+1}}(u) \bigg\|_{\widehat{E}}\\
&\lesssim
\bigg\|\sum_{k\in\Z}\frac{\b(u)}{\|a\|_{\widetilde{F}(0,u)}}\Big\|a(t)\|f\|_{ \widetilde{G}(t,u)}\Big\|_{\widetilde{F}(0,u)}\chi_{I_{k+1}}(u) \bigg\|_{\widehat{E}}\\
&=
\bigg\|\frac{\b(u)}{\|a\|_{\widetilde{F}(0,u)}}\Big\|a(t)\|f\|_{ \widetilde{G}(t,u)}\Big\|_{\widetilde{F}(0,u)} \bigg\|_{\widehat{E}}.
\end{align*}
This concludes the proof of i).

The proof of ii) can be done in a similar vein. Indeed, arguing as  in \eqref{e34} we have 
\begin{align*}
\big \|\b(u)  \|f\|_{ \widetilde{G}(u,\infty)} \big\|_{\widehat{E}}&=
\Big \|\sum_{k\in\Z}\b(u)  \|f\|_{ \widetilde{G}(u,\infty)}\chi_{I_{k+1}}(u) \Big\|_{\widehat{E}}\\
&\lesssim\Big \|\sum_{k\in\Z}\b(\lambda_k) \|f\|_{ \widetilde{G}(\lambda_{k},\infty)}\chi_{I_{k+1}}(u) \Big\|_{\widehat{E}}\\
&\leq\Big \|\sum_{k\in\Z}\b(\lambda_k)\Big(\sum_{j\geq k}\|f\|_{ \widetilde{G}(\lambda_{j},\lambda_{j+1})}\Big)\chi_{I_{k+1}}(u) \Big\|_{\widehat{E}}\\
&= \Big \|\Big(\b(\lambda_k) \sum_{j\geq k}\|f\|_{ \widetilde{G}(\lambda_{j},\lambda_{j+1})}\Big)_{k\in\Z}\Big\|_{d\widehat{E}}.
\end{align*}
This time, $\rho_{B_\infty}<0<\pi_{B_0}$. Then, by Lemma \ref{lema29} ii), there exists an  equivalente function of $\b$, that again we denote  in the same way, such that $\inf_{k\in\Z}\frac{\b(\lambda_{k+1})}{\b(\lambda_k)}>1$ and using Lemma \ref{lema8} ii) we obtain the relation
$$
\Big \|\Big(\b(\lambda_k) \sum_{j\geq k}\|f\|_{ \widetilde{G}(\lambda_{j},\lambda_{j+1})}\Big)_{k\in\Z}\Big\|_{d\widehat{E}}\sim
\Big \|\Big(\b(\lambda_k) \|f\|_{ \widetilde{G}(\lambda_{k},\lambda_{k+1})}\Big)_{k\in\Z}\Big\|_{d\widehat{E}}.$$
Hence,
\begin{align*}
\Big \|\b(u)  \|f\|_{ \widetilde{G}(u,\infty)} \Big\|_{\widehat{E}}&\lesssim
\Big\|\Big(\b(\lambda_k)\|f\|_{ \widetilde{G}(\lambda_{k},\lambda_{k+1})}\Big)_{k\in\Z} \Big\|_{d\widehat{E}}\\
&=\Big\|\sum_{k\in\Z}\b(\lambda_k)\|f\|_{ \widetilde{G}(\lambda_{k},\lambda_{k+1})}\chi_{I_{k}}(u) \Big\|_{\widehat{E}}\\
&\sim\Big\|\sum_{k\in\Z}\frac{\b(u)}{\|a\|_{\widetilde{F}(u,\infty)}}\|a\|_{\widetilde{F}(u,\infty)}\|f\|_{ \widetilde{G}(\lambda_{k},\lambda_{k+1})}\chi_{I_{k}}(u) \Big\|_{\widehat{E}}.
\end{align*}
By Lemma \ref{lema2.3} ii) the function  $u\leadsto \|a\|_{\widetilde{F}(u,\infty)}$ is  slowly varying with $\|a\|_{\widetilde{F}(u,\infty)}\sim \|a\|_{\widetilde{F}(u^2,\infty)}$. Then, for any $u\in I_{k}$ we have that $\|a\|_{\widetilde{F}(u,\infty)}\sim\|a\|_{\widetilde{F}(\lambda_{k+1},\infty)}$, $k\in\Z$, and
\begin{align*}
\|a\|_{\widetilde{F}(\lambda_{k+1},\infty)}\|f\|_{ \widetilde{G}(\lambda_{k},\lambda_{k+1})}&=
\big\|a(t)\|f\|_{ \widetilde{G}(\lambda_{k},\lambda_{k+1})}\big\|_{\widetilde{F}(\lambda_{k+1},\infty)}\lesssim\Big\|a(t)\|f\|_{ \widetilde{G}(\lambda_{k},t)}\Big\|_{\widetilde{F}(\lambda_{k+1},\infty)}\\&\leq\Big\|a(t)\|f\|_{ \widetilde{G}(\lambda_{k},t)}\Big\|_{\widetilde{F}(\lambda_{k},\infty)}.
\end{align*}
Therefore,
\begin{align*}
\Big \|\b(u)  \|f\|_{ \widetilde{G}(u,\infty)} \Big\|_{\widehat{E}}&\lesssim
\bigg\|\sum_{k\in\Z}\frac{\b(u)}{\|a\|_{\widetilde{F}(u,\infty)}}\Big\|a(t)\|f\|_{ \widetilde{G}(\lambda_{k},t)}\Big\|_{\widetilde{F}(\lambda_{k},\infty)}\chi_{I_{k}}(u) \bigg\|_{\widehat{E}}\\
&\lesssim
\bigg\|\sum_{k\in\Z}\frac{\b(u)}{\|a\|_{\widetilde{F}(u,\infty)}}\Big\|a(t)\|f\|_{ \widetilde{G}(u,t)}\Big\|_{\widetilde{F}(u,\infty)}\chi_{I_{k}}(u) \bigg\|_{\widehat{E}}\\
&=
\bigg\|\frac{\b(u)}{\|a\|_{\widetilde{F}(u,\infty)}}\Big\|a(t)\|f\|_{ \widetilde{G}(u,t)}\Big\|_{\widetilde{F}(u,\infty)} \bigg\|_{\widehat{E}}.
\end{align*}
The proof of the lemma is complete.
\end{proof}

\begin{lem}\label{maxI1I2}
Let $E$, $F$, $G$ be r.i. spaces and let $d\widehat{E}$ be the discretization of $\widehat{E}$. Let $a$, $\b\in SV$ such that $a(t^2)\sim a(t)$, $\b(t^2)\sim\b(t)$ and let $A_0$, $A_\infty$, $\B_0$ and $\B_\infty$ be their associated functions defined in (\ref{fb}). The following statements holds:

\begin{itemize}
\item[i)] If $\rho_{A_{\infty}} < \pi_{\varphi_F} \leq \rho_{\varphi_F} < \pi_{A_{0}}$ and $\rho_{B_0}<0<\pi_{B_\infty}$, then
$$\Big \|\b(u) \big\| a(t) \|f\|_{ \widetilde{G}(t,u)}\big\|_{ \widetilde{F}(0,u)} \Big\|_{\widehat{E}}
\sim  \Big\|\Big(\b(\lambda_k)
\|a\|_{\widetilde{F}(0,\lambda_k)}\|f\|_{\widetilde{G}(\lambda_{k-1},\lambda_k)} \Big)_{k\in\Z}\Big\|_{d\widehat{E}}$$
for any measurable function $f$ on $(0,\infty)$.
\item[ii)] If $\rho_{A_0} < \pi_{\varphi_F} \leq \rho_{\varphi_F} < \pi_{A_\infty}$ and $\rho_{B_\infty}<0<\pi_{B_0}$, then
$$\Big \|\b(u) \big\| a(t) \|f\|_{ \widetilde{G}(u,t)}\big\|_{ \widetilde{F}(u,\infty)} \Big\|_{\widehat{E}}
\lesssim  \Big\|\Big(\b(\lambda_k)
\|a\|_{\widetilde{F}(\lambda_{k},\infty)}\|f\|_{\widetilde{G}(\lambda_{k-1},\lambda_k)} \Big)_{k\in\Z}\Big\|_{d\widehat{E}}$$
for any measurable function $f$ on $(0,\infty)$.
\end{itemize}
\end{lem}
\begin{proof}
We start with (i). Arguing as in \eqref{e34} it follows that
\begin{align*}
I:=\Big \|\b(u) \big\| a(t) \|f\|_{ \widetilde{G}(t,u)}\big\|_{ \widetilde{F}(0,u)} \Big\|_{\widehat{E}}&
\sim \Big \|\sum_{k\in\Z}\b(\lambda_k) \big\| a(t) \|f\|_{ \widetilde{G}(t,u)}\big\|_{ \widetilde{F}(0,u)} \chi_{I_k}(u) \Big\|_{\widehat{E}}\\
&\leq \Big \|\sum_{k\in\Z}\b(\lambda_k) \big\| a(t) \|f\|_{ \widetilde{G}(t,\lambda_k)}\big\|_{ \widetilde{F}(0,\lambda_k)} \chi_{I_k}(u) \Big\|_{\widehat{E}}\\
&=\Big\|\Big(\b(\lambda_k) \big\| a(t) \|f\|_{ \widetilde{G}(t,\lambda_k)}\big\|_{ \widetilde{F}(0,\lambda_k)}\Big)_{k\in\Z}\Big\|_{d\widehat{E}}.
\end{align*}
Moreover, for the factor $\| a(t) \|f\|_{ \widetilde{G}(t,\lambda_k)}\|_{ \widetilde{F}(0,\lambda_k)}$ the following estimate holds
\begin{align}
\big\| a(t) \|f\|_{ \widetilde{G}(t,\lambda_k)}\big\|_{ \widetilde{F}(0,\lambda_k)}&\leq \sum_{m\leq k}
\big\| a(t) \|f\|_{ \widetilde{G}(t,\lambda_k)}\big\|_{ \widetilde{F}(\lambda_{m-1},\lambda_m)}\label{e89}\\
&\leq \sum_{m\leq k}
\| a\|_{ \widetilde{F}(\lambda_{m-1},\lambda_m)} \|f\|_{\widetilde{G}(\lambda_{m-1},\lambda_k)}\nonumber\\
&\leq \sum_{m\leq k}
\| a\|_{ \widetilde{F}(\lambda_{m-1},\lambda_m)} \sum_{j=m}^k\|f\|_{\widetilde{G}(\lambda_{j-1},\lambda_j)}.\nonumber
\end{align}
Hence,
$$I\lesssim\Big\|\Big(\b(\lambda_k) \sum_{m\leq k}
\| a\|_{ \widetilde{F}(\lambda_{m-1},\lambda_m)} \sum_{j=m}^k\|f\|_{\widetilde{G}(\lambda_{j-1},\lambda_j)}\Big)_{k\in\Z}\Big\|_{d\widehat{E}}.$$
Moreover, due to Lemma \ref{lema2.3} i) ($\rho_{A_{\infty}} < \pi_{\varphi_F} \leq \rho_{\varphi_F} < \pi_{A_{0}}$) we have that
\begin{equation}\label{e25}
\| a\|_{ \widetilde{F}(\lambda_{m-1},\lambda_m)}\lesssim \| a\|_{ \widetilde{F}(0,\lambda_m)}\sim a(\lambda_m)\varphi_{F}(\ell(\lambda_m)).
\end{equation}
Let us denote by $\tilde{a}$ the slowly varying function $u\leadsto  a(u)\varphi_{F}(\ell(u))$, $u>0$. This function satisfies that $\tilde{a}(u^2)  \sim\tilde{a}(u)$ and its associated functions are $\tilde{A}_0(u)=A_0(u)\varphi_{F}(1/u)$, $\tilde{A}_\infty(u)=A_\infty(u)\varphi_{F}(1/u)$, $u\in(0,1]$. Using the properties of the extension indices it is clear that
$$\rho_{\tilde{A}_{\infty}}\leq \rho_{A_{\infty}}-\pi_{\varphi_F}<0 < \pi_{A_0}-\rho_{\varphi_F}\leq \pi_{\tilde{A}_{0}}.$$
Hence, by Lemma \ref{lema29} ii), there exists an equivalent function to $\tilde{a}$, that we denote in the same way, such that $\inf_{k\in\Z}\frac{\tilde{a}(\lambda_{k+1})}{\tilde{a}(\lambda_k)}>1$. Applying Lemma \ref{lema8} ii), with $\textbf{e}=\ell_1$, we derive 
$$\sum_{m\leq k}
 \tilde{a}(\lambda_m)\sum_{j=m}^k\|f\|_{\widetilde{G}(\lambda_{j-1},\lambda_j)}\lesssim
\sum_{m\leq k}
 \tilde{a}(\lambda_m)\|f\|_{\widetilde{G}(\lambda_{m-1},\lambda_m)}.$$
Therefore, by \eqref{e25}, we obtain
$$I\lesssim \Big\|\Big(\b(\lambda_k)\sum_{m\leq k}
 \| a\|_{ \widetilde{F}(0,\lambda_m)}\|f\|_{\widetilde{G}(\lambda_{m-1},\lambda_m)} \Big)_{k\in\Z}\Big\|_{d\widehat{E}}.$$
Since $\rho_{B_0}<0<\pi_{B_\infty}$, we can apply again Lemmas \ref{lema29} and \ref{lema8} i) with $\textbf{e}=d\widehat{E}$ to deduce that
\begin{align*}
I&\lesssim \Big\|\Big(\b(\lambda_k)\| a\|_{ \widetilde{F}(0,\lambda_k)}
 \|f\|_{\widetilde{G}(\lambda_{k-1},\lambda_k)} \Big)_{k\in\Z}\Big\|_{d\widehat{E}}.
\end{align*}

In order to prove the reverse inequality ``$\gtrsim$'' we proceed as follows.  An argument similar to that of \eqref{e34} yields
\begin{align*}
I&\sim\Big \|\sum_{k\in\Z}\b(\lambda_k) \big\| \a(t) \|f\|_{ \widetilde{G}(t,u)}\big\|_{ \widetilde{F}(0,u)}\chi_{I_{k+1}}(u) \Big\|_{\widehat{E}} \\
&\gtrsim\Big \|\sum_{k\in\Z}\b(\lambda_k) \big\| \a(t) \|f\|_{ \widetilde{G}(t,\lambda_k)}\big\|_{ \widetilde{F}(0,\lambda_k)}\chi_{I_{k+1}}(u) \Big\|_{\widehat{E}}\\
&=\Big \|\Big(\b(\lambda_k) \big\| \a(t) \|f\|_{ \widetilde{G}(t,\lambda_k)}\big\|_{ \widetilde{F}(0,\lambda_k)}\Big) \Big\|_{d\widehat{E}}.
\end{align*}
Moreover
\begin{align*}
\big\| \a(t) \|f\|_{ \widetilde{G}(t,\lambda_k)}\big\|_{ \widetilde{F}(0,\lambda_k)}
&\geq \big\| \a(t) \|f\|_{ \widetilde{G}(t,\lambda_k)}\big\|_{ \widetilde{F}(0,\lambda_{k-1})}\geq\big\| \a(t) \|f\|_{ \widetilde{G}(\lambda_{k-1},\lambda_k)}\big\|_{ \widetilde{F}(0,\lambda_{k-1})}\\
&\sim\| \a\|_{ \widetilde{F}(0,\lambda_k)}\|f\|_{ \widetilde{G}(\lambda_{k-1},\lambda_k)}.
\end{align*}
Hence
$$I\gtrsim \Big \|\Big(\b(\lambda_k) \| \a\|_{ \widetilde{F}(0,\lambda_k)}\|f\|_{ \widetilde{G}(\lambda_{k-1},\lambda_k)}\Big) \Big\|_{d\widehat{E}}$$
and the proof of i) is complete.

Now we proceed with the proof of ii). Similar arguments to the ones given in \eqref{e34} allow to obtain that
\begin{align*}
II:&=\Big \|\b(u) \big\| a(t) \|f\|_{ \widetilde{G}(u,t)}\big\|_{ \widetilde{F}(u,\infty)} \Big\|_{\widehat{E}}\\&\lesssim\Big\|\Big(\b(\lambda_k) \big\| a(t) \|f\|_{ \widetilde{G}(\lambda_{k-1},t)}\big\|_{ \widetilde{F}(\lambda_{k-1},\infty)}\Big)_{k\in\Z}\Big\|_{d\widehat{E}}.
\end{align*}
%\begin{align*}
%\Big \|\b(u) \big\| a(t) \|f\|_{ \widetilde{G}(u,t)}\big\|_{ \widetilde{F}(u,\infty)} \Big\|_{\widehat{E}}&\sim \Big \|\sum_{k\in\Z}\b(\lambda_k) \big\| a(t) \|f\|_{ \widetilde{G}(u,t)}\big\|_{ \widetilde{F}(u,\infty)} \chi_{I_k}(u) \Big\|_{\widehat{E}}\\
%&\leq \Big \|\sum_{k\in\Z}\b(\lambda_k) \big\| a(t) \|f\|_{ \widetilde{G}(\lambda_{k-1},t)}\big\|_{ \widetilde{F}(\lambda_{k-1},\infty)} \chi_{I_k}(u) \Big\|_{\widehat{E}}\\
%&\sim\Big\|\Big(\b(\lambda_k) \big\| a(t) \|f\|_{ \widetilde{G}(\lambda_{k-1},t)}\big\|_{ \widetilde{F}(\lambda_{k-1},\infty)}\Big)_{k\in\Z}\Big\|_{\textbf{e}}.
%\end{align*}
And arguing as in \eqref{e89} 
\begin{align*}
\big\| a(t) \|f\|_{ \widetilde{G}(\lambda_{k-1},t)}\big\|_{ \widetilde{F}(\lambda_{k-1},\infty)}&\leq \sum_{m\geq k}
\big\| a(t) \|f\|_{ \widetilde{G}(\lambda_{k-1},\lambda_m)}\big\|_{ \widetilde{F}(\lambda_{m-1},\lambda_m)}\\
&\leq \sum_{m\geq k}
\| a\|_{ \widetilde{F}(\lambda_{m-1},\lambda_m)} \|f\|_{\widetilde{G}(\lambda_{k-1},\lambda_m)}\\
&\leq \sum_{m\geq k}
\| a\|_{ \widetilde{F}(\lambda_{m-1},\lambda_m)} \sum_{j=k}^m\|f\|_{\widetilde{G}(\lambda_{j-1},\lambda_j)}.
\end{align*}
Lemma \ref{lema2.3} ($\rho_{A_{0}} < \pi_{\varphi_F} \leq \rho_{\varphi_F} < \pi_{A_{\infty}}$)
establishes that
$$\| a\|_{ \widetilde{F}(\lambda_{m-1},\lambda_m)}\leq\| a\|_{ \widetilde{F}(\lambda_{m-1},\infty)}\sim a(\lambda_{m})\varphi_{F}(\ell(\lambda_m)).$$ We denote again by $\tilde{a}$ the slowly varying function $u\leadsto a(u)\varphi_{F}(\ell(u))$, $u>0$. Its associated functions are $\tilde{A}_0(u)=A_0(u)\varphi_{F}(1/u)$, $\tilde{A}_\infty(u)=A_\infty(u)\varphi_{F}(1/u)$ and this time
$$\rho_{\tilde{A}_{0}}\leq \rho_{A_0}-\pi_{\varphi_F}<0 < \pi_{A_\infty}-\rho_{\varphi_F}\leq \pi_{\tilde{A}_{\infty}}.$$
Then, there exist an equivalent function, that we denote in the same way, such that $\sup_{k\in\Z}\frac{\tilde{a}(\lambda_{k+1})}{\tilde{a}(\lambda_k)}<1$. Using Lemma \ref{lema8}  i) with $\textbf{e}=\ell_1$, we have 
$$\sum_{m\geq k}
 \tilde{a}(\lambda_m)\sum_{j=k}^m\|f\|_{\widetilde{G}(\lambda_{j-1},\lambda_j)}\lesssim
\sum_{m\geq k}
 \tilde{a}(\lambda_m)\|f\|_{\widetilde{G}(\lambda_{m-1},\lambda_m)}.$$
All together, 
$$II\lesssim \Big\|\Big(\b(\lambda_k)\sum_{m\geq k}
\tilde{a}(\lambda_m) \|f\|_{\widetilde{G}(\lambda_{m-1},\lambda_m)} \Big)_{k\in\Z}\Big\|_{d\widehat{E}}.$$
Finally, by Lemmas \ref{lema29} y \ref{lema8} with $\textbf{e}=d\widehat{E}$ ($\rho_{B_\infty}<0<\pi_{B_0}$) we obtain the desired inequality
$$II\lesssim \Big\|\Big(\b(\lambda_k)
\|a\|_{\widetilde{F}(\lambda_k,\infty)} \|f\|_{\widetilde{G}(\lambda_{k-1},\lambda_k)} \Big)_{k\in\Z}\Big\|_{d\widehat{E}}.$$
\end{proof}

\section{Interpolation Methods} \label{interpolation methods}

In the sequel, $\overline{X}=(X_{0}, X_{1})$ will be a \textit{compatible (quasi-) Banach couple}.
Next, we collect, following \cite{FMS-1}, the necessary definitions and statements dealing with the real interpolation methods $\overline{X}_{\theta,\b,E}$, $\overline{X}_{\theta,\b,E,a,F}^{\mathcal R}$, $\overline{X}_{\theta,\b,E,a,F}^{\mathcal L}$, involving slowly varying functions and r.i. spaces.

\begin{defn}\label{defrealmethod}
Let $E$ be an r.i.  space,  $\b\in SV$ and $0\leq\theta \leq 1$. The real interpolation space $\overline{X}_{\theta,\b,E}\equiv(X_0,X_1)_{\theta, \b, E}$ consists of all $f$ in $X_{0} + X_{1}$ for which
$$ \|f\|_{\theta,\b,E} := \big \| t^{-\theta} {\b}(t) K(t,f) \big \|_{\widetilde{E}} < \infty.$$
\end{defn}

It is a well-known fact that  $\overline{X}_{\theta, \b,E}$ is a (quasi-) Banach space, and it is intermediate for the couple $\overline{X}$, that is,
$$X_0\cap X_1\hookrightarrow \overline{X}_{\theta, \b,E}\hookrightarrow X_0+X_1,$$
provided that 
$0< \theta< 1$,
 $\theta=0$ and $\|\b\|_{\widetilde{E}(1,\infty)}<\infty$ or
 $\theta=1$ and $\|\b\|_{\widetilde{E}(0,1)}<\infty$.
If none of these conditions holds, the space is trivial, that is $\overline{X}_{\theta, \b,E}=\{0\}$.

When $E=L_q$ and $\b \equiv 1$, then $\overline{X}_{\theta,\b,E}$ coincides with the classical real interpolation space $\overline{X}_{\theta,q}$. 

The reiteration spaces
$$(\overline{X}_{\theta_0,\b_0,E_0},\overline{X}_{\theta_1,\b_1,E_1})_{\theta,\b,E},$$ with $\theta_0<\theta_1,$ have been studied in detail in \cite{FMS-1,FMS-2,FMS-3} for general r.i. spaces $E$, and in \cite{AEEK,GOT} for $E=L_q$, $0<q\leq \infty$. For other special cases see e.g. \cite{CSe-2,Do,EO,EOP,tesisalba}. When $\theta=0,1$ the resulting reiteration spaces do not belong to the same scale and the ${\mathcal R}$ and ${\mathcal L}$ constructions are needed to describe them.

\begin{defn}
Let $E$, $F$ be two r.i. spaces, $a, \b\in SV$ and $0\leq \theta\leq 1$. 
The (quasi-) Banach space 
$\overline{X}_{\theta,\b,E,a,F}^{\mathcal R}\equiv(X_0,X_1)_{\theta,\b,E,a,F}^{\mathcal R}$ consists of all $f\in X_0+X_1$ for which
$$\| f  \|_{\mathcal{R};\theta,\b,E,a,F} := \Big \|  \b(t) \|   s^{-\theta} a(s) K(s,f) \|_{\widetilde{F}(t,\infty)}      \Big   \|_{\widetilde{E}}<\infty.$$
\end{defn}

The space $\mathcal{R}$ is intermediate for the couple $\overline{X}$, that is,
$$ X_0\cap X_1 \hookrightarrow  \overline{X}_{\theta,\b,E,a,F}^{\mathcal R} \hookrightarrow X_0+X_1$$
provided that any of the following conditions holds:
\begin{enumerate}
\item[1.] $0 < \theta < 1$ and  $\|\b\|_{\widetilde{E}(0,1)} <\infty$, 
\item[2.] $\theta =0$,  $\|\b\|_{\widetilde{E}(0,1)} \!<\infty$ and   $\big \| \b(t)\|a\|_{\widetilde{F}(t,\infty)}  \big \|_{\widetilde{E}(1,\infty)} \!< \infty$ or
\item[3.] $\theta =1$, $ \|\b\|_{\widetilde{E}(0,1)} \!<\infty$,    $\big \| \b(t) \|a\|_{\widetilde{F}(t,1)} \big \|_{\widetilde{E}(0,1)} \!< \infty$ and  $\|\a\b\|_{\widetilde{E}(0,1)}<\infty$.
\end{enumerate}
Otherwise, $\overline{X}^{\mathcal R}_{\theta, \b,E,a,F}=\{0\}$.

\begin{defn}\label{defLR}
 Let $E$, $F$ be two r.i. spaces, $a, \b\in SV$ and $0\leq \theta\leq 1$. The space $\overline{X}_{\theta,\b,E,a,F}^{\mathcal L}\equiv(X_0,X_1)_{ \eta,\b,E,a,F}^{\mathcal L}$ consists of all $f \in X_{0} + X_{1}$ for which 
$$ \|f \|_{\mathcal L;\theta,\b,E,a,F}:=
\Big\|\b(t) \|s^{-\theta} a(s) K(s,f) \|_{\widetilde{F}(0,t)}\Big\|_{\widetilde{E}} < \infty. $$
\end{defn}
This is  a (quasi-) Banach space. Moreover, it is intermediate for the couple $\overline{X}$, $$ X_0\cap X_1 \hookrightarrow  \overline{X}_{\theta,\b,E,a,F}^{\mathcal L} \hookrightarrow X_0+X_1,$$ provided that:
\begin{enumerate}
\item[1.] $0 < \theta < 1$  and  $\|\b\|_{\widetilde{E}(1,\infty)} <\infty$,  
\item[2.] $\theta =0$, $ \big \| \b(t)\|a\|_{\widetilde{F}(1,t)}  \big \|_{\widetilde{E}(1,\infty)} \!< \infty$  or 
\item[3.] $\theta =1$, $\|\b\|_{\widetilde{E}(1,\infty)} \!<\infty$ and $\big \| \b(t) \|a\|_{\widetilde{F}(0,t)} \big \|_{\widetilde{E}(0,1)} \!< \infty$.
\end{enumerate}
If none of these conditions holds, then $\overline{X}^{\mathcal L}_{\theta, \b,E,a,F}$ is  the trivial space.

The spaces $\overline{X}_{\theta,\b,\widehat{E},a,F}^{\mathcal R}$ and $\overline{X}_{\theta,\b,\widehat{E},a,F}^{\mathcal L}$ can be defined analogously replacing $\widetilde{E}$ by $\widehat{E}$ in previous definitions.

We refer to the recent papers \cite{Do2020,Do2021,FMS-RL1,FMS-RL2,FMS-RL3} for reiteration theorems for couples formed by arbitrary combinations of the previous spaces under the condition that the parameters $\theta_0$ and $\theta_1$ are not equal.
Again, in the extremal cases $\theta=0, 1$ the resulting reiteration spaces belong to some extremal constructions introduced in \cite{Do2021,FMS-RL1,FMS-RL2}. Let us recall the definition of the interpolation methods $(\mathcal R,\mathcal L)$ and $(\mathcal L,\mathcal R)$.
\begin{defn}\label{defLRR^}
Let $E$, $F$, $G$ be  r.i. spaces, $\a, \b, c \in SV$ and $0< \theta<1$. 
The space $\overline{X}_{\theta,c,\widehat{E},\b,F,\a,G}^{\mathcal R,\mathcal L}\equiv(X_0,X_1)_{ \eta,c,\widehat{E},\b,F,\a,G}^{\mathcal R,\mathcal L}$ is the set of all $f\in X_0+X_1$ for which
\begin{equation}\label{dRL}
 \|f \|_{\mathcal R,\mathcal L;\theta,c,\widehat{E},\b,F,a,G} :=
\bigg\|c(u)\Big\|\b(t) \|s^{-\theta} a(s) K(s,f) \|_{\widetilde{G}(t,u)}\Big\|_{\widetilde{F}(0,u)}\bigg\|_{\widehat{E}}<\infty.
\end{equation}

The space $\overline{X}_{\theta,c,\widehat{E},\b,F,a,G}^{\mathcal L,\mathcal R}\equiv(X_0,X_1)_{ \eta,c,\widehat{E},\b,F,a,G}^{\mathcal L,\mathcal R}$ is the set of all $f\in X_0+X_1$ such that
\begin{equation}\label{dLR}
 \|f \|_{\mathcal L,\mathcal R;\theta,c,\widehat{E},\b,F,\a,G} :=
\bigg\|c(u)\Big\|\b(t) \|s^{-\theta} a(s) K(s,f) \|_{\widetilde{G}(u,t)}\Big\|_{\widetilde{F}(u,\infty)}\bigg\|_{\widehat{E}}<\infty.
\end{equation}
\end{defn}

Lemma \ref{key1} allows  us to obtain the following inclusions:

\begin{cor}\label{cor45}
Let $E$, $F$, $G$ be r.i. spaces, let  $a$, $\b$, $c\in SV$ such that $\b(t^2)\sim \b(t)$, $c(t^2)\sim c(t)$ and let $B_0$, $B_\infty$, $C_0$ and $C_\infty$ be their respective associated functions defined by (\ref{fb}). The following statements holds:
\begin{itemize}
\item[i)] If $\rho_{B_{\infty}} < \pi_{\varphi_F} \leq \rho_{\varphi_F} < \pi_{B_{0}}$ and $\rho_{C_0}<0<\pi_{C_\infty}$, then
$$\overline{X}_{\theta,d,\widehat{E},\b,F,\a,G}^{\mathcal R,\mathcal L}=\overline{X}_{\theta,c,\widehat{E},a,G}^{\mathcal L}$$
where $d(t)=c(t)/\|\b\|_{\widetilde{F}(0,u)}$, $u>0$.
\item[ii)] If $\rho_{B_0} < \pi_{\varphi_F} \leq \rho_{\varphi_F} < \pi_{B_\infty}$ and $\rho_{C_\infty}<0<\pi_{C_0}$, then
$$\overline{X}_{\theta,d,\widehat{E},\b,F,a,G}^{\mathcal L,\mathcal R}=\overline{X}_{\theta,c,\widehat{E},a,G}^{\mathcal R}$$
where $d(t)=c(t)/\|\b\|_{\widetilde{F}(u,\infty)}$, $u>0$.
\end{itemize}
\end{cor}
%\textcolor{blue}{\begin{lem}\cite[Lemma 2.10]{FMS-3}\label{lemLRK}
%Let $E, F$ be  r.i. spaces, $a, \b\in SV$ and $0\leq \eta\leq 1$. Then, for all $f\in X_0+X_1$ and $u>0$
%\begin{equation}\label{e5}
%u^{-\theta}\a(u)\|\b\|_{\widetilde{E}(0,u)}K(u,f)\lesssim \Big\| \b(t) \|s^{- \eta} \a(s) K(s,f)\|_{ \widetilde{F}(t,u)}\Big\|_{ \widetilde{E}(0,u)}
%\end{equation}
%and 
%\begin{equation}\label{e7}
%u^{-\theta}\a(u)\|\b\|_{\widetilde{E}(u,\infty)}K(u,f)\lesssim \Big\| \b(t) \|s^{- \eta} \a(s) K(s,f)\|_{ \widetilde{F}(u,t)}\Big\|_{ \widetilde{E}(u,\infty)}.
%\end{equation}
%\end{lem}}

%%%%%%%%%%%%%%%%%%%%%%%%%%%%%%%%%%%%%%%%%%%%%%%%%%%%%%%%%%%%%%%%%%%%%%%%%%%%%%%%%%%%%%%%%%%%%%%%%%%%%%%%%%%%%%%%%%%%%%%%%%%%%%%%%%%%%%%%%%%

\section{Reiteration theorem}\label{S5}

In this section we shall prove the main result of this paper,  the characterization of the interpolation space
$$(\overline{X}^{\mathcal R}_{\theta, \b_0,E_0,\a,F}, \overline{X}^{\mathcal L}_{\theta,\b_1,E_1,\a,F})_{\eta,\b,E}$$
for all possible values of $\eta\in[0,1]$. To this end we will additionally need a generalized Holsmtedt type formula and a change of variables.

\begin{thm}\label{Holmstedt}
Let $0<\theta<1$. Let $E_0$, $E_1$, $F$ be  r.i. spaces; $\a$, $\b_0$, $\b_1\in SV$ with  $\|\b_0\|_{\widetilde{E}_0(0,1)}<\infty$ and $\|\b_1\|_{\widetilde{E}_1(1,\infty)}<\infty$. Then, for every $f\in \overline{X}^{\mathcal R}_{\theta, \b_0,E_0,\a,F}+ \overline{X}^{\mathcal L}_{\theta,\b_1,E_1,a,F}$ and $u>0$
 \begin{align*}
 K\big( \phi(u), f; \overline{X}^{\mathcal R}_{\theta, \b_0,E_0,\a,F}, \overline{X}^{\mathcal L}_{\theta,\b_1,E_1,\a,F}\big )& \sim\Big\| \b_0(t) \|s^{- \theta} \a(s) K(s,f)\|_{ \widetilde{F}(t,u)}\Big\|_{ \widetilde{E}_0(0,u)} \label{H2}\\
& +\phi(u)\Big\| \b_1(t)\|s^{- \theta} \a(s) K(s,f)\|_{\widetilde{F}(u,t)}\Big\|_{ \widetilde{E}_1(u,\infty)},\nonumber
\end{align*}
where
%\begin{equation}\label{rho1}
$$\phi(u) =\frac{\|\b_0\|_{\widetilde{E}_0(0,u)}}{\|\b_1\|_{\widetilde{E}_1(u,\infty)}},\quad  u>0.$$
%\end{equation}
\end{thm}

\begin{proof}

Given $f\in X_0+X_1$ and $u>0$, we consider the following (quasi-) norms
\begin{align*}
(P_0 f)(u) &= \Big\| \b_0(t)\, \|s^{- \theta} \a(s) K(s,f)\|_{ \widetilde{F}(t,u)}\Big\|_{ \widetilde{E}_0(0,u)},\\
(Q_1 f)(u) &= \Big\| \b_1(t)\| s^{- \theta} \a(s) K(s,f)\|_{ \widetilde{F}(u,t)}\Big\|_{\widetilde{E}_1(u,\infty)}.
\end{align*}
Denote $Y_0= \overline{X}^{\mathcal R}_{\theta, \b_0,E_0,\a,F}$ and $Y_1=\overline{X}^{\mathcal L}_{\theta,\b_1,E_1,a,F}$.
The proof of the estimate 
$$K(\phi(u), f;Y_0,Y_1) \lesssim (P_{0}f)(u) + \phi(u) (Q_{1}f)(u)$$ can be done exactly as we did in Theorem 3.3 from \cite{FMS-RL3}. Therefore, we only have to prove the converse estimate, that is, 
\begin{equation}\label{e10}
(P_0f)(u)+\phi(u)(Q_1f)(u)\lesssim K(\phi(u), f;Y_0,Y_1)
\end{equation}
for all $f\in Y_0+Y_1$ and $u>0$.

Fix $u>0$ and let $f=f_0+f_1$ be any decomposition of $f$ where $f_0\in Y_0$ and $f_1\in Y_1$. By the (quasi)-subadditivity of  the $K$-functional and the definition of the norms in $Y_0$ and $Y_1$, we derive
$$(P_0f)(u)\lesssim (P_0f_0)(u)+(P_0f_1)(u)\leq \|f_0\|_{Y_0}+(P_0f_1)(u),$$
$$(Q_1f)(u)\lesssim (Q_1f_0)(u)+(Q_1f_1)(u)\leq (Q_1f_0)(u)+\|f_1\|_{Y_1}.$$
Thus, we have to study the boundedness of $(P_0f_1)(u)+\phi(u)(Q_0f_1)(u)$  by $\|f_0\|_{Y_0}+\phi(u)\|f_1\|_{Y_1}$. Let us start with $(P_0f_1)(u)$. It is easy to observe that
\begin{align*}
(P_0f_1)(u)
&\leq \|\b_0\|_{\widetilde{E}_0(0,u)}\|s^{-\theta}\a(s)K(s,f_1)\|_{\widetilde{F}(0,u)}\\
&=\phi(u)\|\b_1\|_{\widetilde{E}_1(u,\infty)}\|s^{-\theta}\a(s)K(s,f_1)\|_{\widetilde{F}(0,u)}\\
&\leq\phi(u)\Big\|\b_1(t)\|s^{-\theta}\a(s)K(s,f_1)\|_{\widetilde{F}(0,t)}\Big\|_{\widetilde{E}_1(u,\infty)}\\
&\leq  \phi(u)\Big\|\b_1(t)\|s^{-\theta}\a(s)K(s,f_1)\|_{\widetilde{F}(0,t)}\Big\|_{\widetilde{E}_1}=\phi(u)\|f_1\|_{Y_1}.
\end{align*}
Similarly, 
\begin{align*}
\phi(u)(Q_1f_0)(u)
&\leq \|\b_0\|_{\widetilde{E}_0(0,u)}\|s^{-\theta}\a(s)K(s,f_0)\|_{\widetilde{F}(u,\infty)}\\
&\leq\Big\|\b_0(t)\|s^{-\theta}\a(s)K(s,f_0)\|_{\widetilde{F}(t,\infty)}\Big\|_{\widetilde{E}_0(0,u)}\\
&\leq \Big\|\b_0(t)\|s^{-\theta}\a(s)K(s,f_0)\|_{\widetilde{F}(t,\infty)}\Big\|_{\widetilde{E}_0}=\|f_0\|_{Y_0}.
\end{align*}
Putting together the previous  estimates we establish the inequality
$$ (P_0 f)(u)+\phi(u)(Q_1 f)(u)\lesssim  \|f_0\|_{Y_0}+\phi(u)\|f_1\|_{Y_1}.$$
Finally, taking infimum over all possible decomposition of $f=f_0+f_1$, with $f_0\in Y_0$ and $f_1\in Y_1$, we deduce  \eqref{e10}.
\end{proof}

\begin{lem}\label{lemacv0}
Let $E$ be an r.i. space, $0\leq \theta\leq 1$,  $\phi$, $\b\in SV$ such that $\phi(t)\sim \phi(t^2)$ with associated functions $\Phi_{0}$, $\Phi_{\infty}$. If $\rho_{\Phi_\infty}<0<\pi_{\Phi_0}$, then the equivalence
$$\|\phi(t)^{-\theta}\b(\phi(t))K(\phi(t),f)\|_{\widehat{E}}\sim\|t^{-\theta}\b(t)K(t,f)\|_{\widetilde{E}}$$
holds for all $f\in X_0+X_1$, with  constant independent of $f$.
\end{lem}

\begin{proof}
First of all, we observe that by an interpolation argument, it suffices to show the equivalence of the norms for $E=L_1$ and $E=L_\infty$. To this end fix $f\in X_0+X_1$ and denote $H(u)=u^{-\theta}\b(u)K(u,f)$, $u>0$. Observe that, due to the properties of the slowly varying functions and the $K$-functional, if $\varphi\sim \psi$ then $H(\varphi)\sim H(\psi)$.

By hypothesis, the indices of the function $\Phi_0$ are strictly positive and finite, 
hence there exists a smooth function $\Psi_0\sim\Phi_0$ such that $u\Psi_0'(u)\sim\Psi_0(u)$, $0<u\leq 1$, $\Psi_0(1)=\Phi_0(1)$ and $\lim_{u\rightarrow 0}\Psi_0(u)=0$ (see \cite[Lemma 2.1]{PS2}). Then, changing variables, $t=\Psi_0(u)$, and using the fact that $H\circ\Psi_0\sim H\circ\Phi_0$ we establish
\begin{align*}
\int_{0}^{\Phi_0(1)} |H(t)| \frac{dt}{t}& = \int_{0}^{1}|H(\Psi_0(u))| \frac{\Psi_0'(u)}{\Psi_0(u)}\ du 
\sim \int_{0}^{1} |H(\Psi_0(u))| \frac{du}{u}\\&\sim\int_{0}^{1} |H(\Phi_0(u))| \frac{du}{u}.
\end{align*}
Now, another change of variables, $u=1/\ell(t)$, and the fact that $\Phi_0(1/\ell(t))=\phi(t)$, $0<t\leq1$, yield
\begin{equation}\label{i1}
\int_{0}^{\Phi_0(1)} |H(t)| \frac{dt}{t}\sim\int_{0}^{1} |H(\phi(t))| \frac{dt}{t\ell(t)}.
\end{equation}
Similarly, since the indices of the function
$\overline{\Phi}_\infty(u)=\Phi_\infty(1/u)$, $1\leq u<\infty$, are strictly positive and finite, there exist a smooth function $\Psi_\infty\sim\overline{\Phi}_\infty$ such that $u\Psi_\infty'(u)\sim\Psi_\infty(u)$, $1\leq u<\infty$, $\Psi_\infty(1)=\overline{\Phi}_\infty(1)=\Phi_\infty(1)$ and $\lim_{u\rightarrow \infty}\Psi_\infty(u)=\infty$. Consequently, changing variables, $t=\Psi_\infty(t)$, it follows 
$$\int_{\Phi_\infty(1)}^{\infty} |H(t)| \frac{dt}{t} = \int_{1}^{\infty}|H(\Psi_\infty(u))| \frac{\Psi_\infty'(u)}{\Psi_\infty(u)}\ du 
\sim \int_{1}^{\infty} |H(\Psi_\infty(u))| \frac{du}{u}. $$
Using that $H\circ\Psi_\infty\sim H\circ\overline{\Phi}_\infty$, the change of variables $u=\ell(t)$ and the fact that $\overline{\Phi}_\infty(\ell(t))=\phi(t)$, $1\leq t<\infty$, we have
\begin{equation}\label{i2}
\int_{\Phi_\infty(1)}^{\infty} |H(s)| \frac{ds}{s}\sim\int_{1}^{\infty}|H(\overline{\Phi}_\infty(u))| \frac{du}{u}=\int_{1}^{\infty} |H(\phi(t))| \frac{dt}{t\ell(t)}.
\end{equation}
The equivalences \eqref{i1} and \eqref{i2} give
$$\|H\|_{\widetilde{L}_1}\sim \|H\circ\phi\|_{\widehat{L}_1}$$
(observe that  $\Phi_0(1)=\Phi_\infty(1)$).
The same is true for $E=L_\infty$, that is $\|H\|_{L_\infty}\sim \|H\circ\phi\|_{L_\infty}$. Thus, by the interpolation properties of the space $E$, we obtain the inequality
$$ \|H\circ\phi\|_{\widehat{E}}\lesssim\|H\|_{\widetilde{E}}.$$

The reverse inequality can be proved applying the same techniques with inverse functions. 
\end{proof}

Now, we are in position to establish the main interpolation theorem of our paper. 

\begin{thm}\label{maintheorem}
Let $0<\theta<1$. 
Let  $E$, $E_0$, $E_1$, $F$ be  r.i. spaces and let $\a$, $\b$,  $\b_0$, $\b_1\in SV$ be  such that $\b_0(t)\sim\b_0(t^2)$ and $\b_1(t)\sim\b_1(t^2)$.  Assume that $E_0$, $E_1$ and  the associated functions $B_{0,0}$, $B_{0,\infty}$, $B_{1,0}$, $B_{1,\infty}$ of $\b_0$ and $\b_1$, respectively, defined in \eqref{fb} satisfy that
\begin{equation*}\label{cond}
\rho_{B_{0,\infty}}<\pi_{\varphi_{E_0}}\leq \rho_{\varphi_{E_0}}<\pi_{B_{0,0}}\quad\mbox{and}\quad 
\rho_{B_{1,0}}<\pi_{\varphi_{E_1}}\leq \rho_{\varphi_{E_1}}<\pi_{B_{1,\infty}}.
\end{equation*}
Let $0\leq\eta\leq1$ be  a parameter and define
$$B_\eta(u)=\big(\b_0(u)\varphi_{E_0}(\ell(u))\big)^{1-\eta}\;\big(\b_1(u)\varphi_{E_1}(\ell(u))\big)^{\eta}\;\b\Big(\frac{\b_0(u)\varphi_{E_0}(\ell(u))}{\b_1(u)\varphi_{E_1}(\ell(u))}\Big),\quad u>0.$$
\begin{itemize}
\item [a)] If $0<\eta<M_1:=\min\Big\{\Big(1-\frac{\pi_{\B_{1,0}}-\rho_{\varphi_{E_1}}}{\pi_{\B_{0,0}}-\rho_{\varphi_{E_0}}}\Big)^{-1},\Big(1-\frac{\rho_{\B_{1,\infty}}-\pi_{\varphi_{E_1}}}{\rho_{\B_{0,\infty}}-\pi_{\varphi_{E_0}}}\Big)^{-1}\Big\}$, then
$$\left(\overline{X}^{\mathcal R}_{\theta, \b_0,E_0,\a,F}, \overline{X}^{\mathcal L}_{\theta,\b_1,E_1,\a,F}\right)_{\eta,\b,E}=\overline{X}^{\mathcal R}_{\theta,\B_\eta,\widehat{E},a,F}.$$

\item [b)] If $M_2:=\max\Big\{\Big(1-\frac{\rho_{\B_{1,0}}-\pi_{\varphi_{E_1}}}{\rho_{\B_{0,0}}-\pi_{\varphi_{E_0}}}\Big)^{-1},\Big(1-\frac{\pi_{\B_{1,\infty}}-\rho_{\varphi_{E_1}}}{\pi_{\B_{0,\infty}}-\rho_{\varphi_{E_0}}}\Big)^{-1}\Big\}<\eta<1$, then
$$\left(\overline{X}^{\mathcal R}_{\theta, \b_0,E_0,\a,F}, \overline{X}^{\mathcal L}_{\theta,\b_1,E_1,\a,F}\right)_{\eta,\b,E}=\overline{X}^{\mathcal L}_{\theta,\B_\eta,\widehat{E},a,F}.$$

\item[c)] If $M_1\leq \eta\leq M_2$, then
$$\left(\overline{X}^{\mathcal R}_{\theta, \b_0,E_0,\a,F}, \overline{X}^{\mathcal L}_{\theta,\b_1,E_1,\a,F}\right)_{\eta,\b,E}= \overline{X}^{\mathcal L}_{\theta,\B_\eta^\#,\widehat{E},\a^\#,F}$$
where $\B_\eta^\#(u)=\frac{\B_\eta(u)}{\b_0(u)\varphi_{E_0}(\ell(u))}$ and 
$\a^\#(u)=\a(u)\b_0(u)\varphi_{E_0}(\ell(u))$, $u>0$.
\item[d)] If $\|\b\|_{\widetilde{E}(1,\infty)}<\infty$, then
$$\left(\overline{X}^{\mathcal R}_{\theta, \b_0,E_0,\a,F}, \overline{X}^{\mathcal L}_{\theta,\b_1,E_1,\a,F}\right)_{0,\b,E}= \overline{X}^{\mathcal R}_{\theta,\B_0,\widehat{E},a,F}\cap\overline{X}_{\theta,\b\circ\phi,\widehat{E},\b_0,E_0,\a,F}^{\mathcal R,\mathcal L},$$
where $\phi(u)=\frac{\b_0(u)\varphi_{E_0}(\ell(u))}{\b_1(u)\varphi_{E_1}(\ell(u))}$, $u>0$.

\item[e)] If $\|\b\|_{\widetilde{E}(0,1)}<\infty$, then
$$\left(\overline{X}^{\mathcal R}_{\theta, \b_0,E_0,\a,F}, \overline{X}^{\mathcal L}_{\theta,\b_1,E_1,\a,F}\right)_{1,\b,E}= \overline{X}^{\mathcal L}_{\theta,\B_1,\widehat{E},a,F}\cap\overline{X}_{\theta,\b\circ\phi,\widehat{E},\b_1,E_1,\a,F}^{\mathcal L,\mathcal R},$$
where $\phi$ is defined above.
\end{itemize}

\end{thm}

\begin{proof}
Throughout the proof we use the notation $Y_0=\overline{X}^{\mathcal R}_{\theta, \b_0,E_0,\a,F}$, $Y_1=\overline{X}^{\mathcal L}_{\theta,\b_1,E_1,\a,F}$, $\overline{K}(u,f)=K(u,f;Y_0,Y_1)$, $u>0$, and 
$$\phi(u)=\frac{\|\b_0\|_{\widetilde{E}_0(0,u)}}{\|\b_1\|_{\widetilde{E}_1(u,\infty)}},\quad u>0.$$
It is clear that $\phi$ is an increasing slowly varying function. Moreover, from Lemma \ref{lema2.3}, we have
\begin{equation}\label{e82}
\phi(u)\sim\frac{\b_0(u)\varphi_{E_0}(\ell(u))}{\b_1(u)\varphi_{E_1}(\ell(u))},\quad u>0,
\end{equation}
$\phi(u)\sim \phi(u^2)$ and its associated functions, in the sense of \eqref{fb}, are
$$\Phi_0(v)=\frac{\B_{0,0}(v)\varphi_{E_0}(1/v)}{\B_{1,0}(v)\varphi_{E_1}(1/v)},\qquad\Phi_\infty(v)=\frac{\B_{0,\infty}(v)\varphi_{E_0}(1/v)}{\B_{1,\infty}(v)\varphi_{E_1}(1/v)},\quad 0<v\leq 1.$$ 
By  the properties of the extension indices, it holds $$\rho_{\Phi_\infty}\leq\rho_{\B_{0,\infty}}-\pi_{\varphi_{E_0}}-\pi_{\B_{1,\infty}}+\rho_{\varphi_{E_1}}<0<\pi_{\B_{0,0}}-\rho_{\varphi_{E_0}}-\rho_{\B_{1,0}}+\pi_{\varphi_{E_1}}\leq \pi_{\Phi_0}.$$
Then, Lemma \ref{lemacv0} establishes the equivalence
$$
\|f\|_{\overline{Y}_{\eta,\b,E}}=\|u^{-\eta}\b(u)\overline{K}(u,f)\|_{\widetilde{E}}\sim \|\phi(u)^{-\eta}\b(\phi(u))\overline{K}(\phi(u),f)\|_{\widehat{E}}.$$
Applying generalized Holmstedt type formula, Theorem \ref{Holmstedt}, and the triangular inequality we obtain that 
\begin{equation}\label{ei12}
\max(I_1,I_2)\leq \|f\|_{\overline{Y}_{\eta,\b,E}}\leq I_1+I_2\quad \mbox{for all}\quad 0\leq \eta\leq 1.
\end{equation}
where
$$I_1:=\Big \|\phi(u)^{-\eta} \b(\phi(u)) \big\| \b_0(t) \|s^{- \theta} \a(s) K(s,f)\|_{ \widetilde{F}(t,u)}\big\|_{ \widetilde{E}_0(0,u)} \Big\|_{\widehat{E}}$$
and 
$$I_2:=\Big\|\phi(u)^{1-\eta} \b(\phi(u)) \big\| \b_1(t)\|s^{- \theta} \a(s) K(s,f)\|_{\widetilde{F}(u,t)}\big\|_{ \widetilde{E}_1(u,\infty)}\Big \|_{\widehat{E}}.$$
Therefore, in order to identify  the space $\overline{Y}_{\eta,\b,E}$ we have to estimate $I_1$ and $I_2$.
First, we proceed with the estimates from above of both quantities.
It is clear by \eqref{e82} and the definition of $B_\eta$ that
\begin{align}
I_1&\leq  \Big \|\phi(u)^{-\eta} \b(\phi(u))\|\b_0\|_{ \widetilde{E}_0(0,u)}  \|s^{- \theta} \a(s) K(s,f)\|_{ \widetilde{F}(0,u)} \Big\|_{\widehat{E}}\label{eI1}\\
&\sim \Big \|\B_\eta(u) \|s^{- \theta} \a(s) K(s,f)\|_{ \widetilde{F}(0,u)} \Big\|_{\widehat{E}}=\|f\|_{\overline{X}^{\mathcal L}_{\theta,\B_\eta,\widehat{E},a,F}},\nonumber
\end{align}
and
\begin{align}
I_2&\leq \Big\|\phi(u)^{1-\eta} \b(\phi(u)) \|\b_1\|_{ \widetilde{E}_1(u,\infty)}\|s^{- \theta} \a(s) K(s,f)\|_{\widetilde{F}(u,\infty)}\Big \|_{\widehat{E}}\label{eI2}\\
&\sim\Big \|\B_\eta(u) \|s^{- \theta} \a(s) K(s,f)\|_{ \widetilde{F}(u,\infty)} \Big\|_{\widehat{E}}=\|f\|_{\overline{X}^{\mathcal R}_{\theta,\B_\eta,\widehat{E},a,F}}.\nonumber
\end{align}
Hence  $$I_1+I_2\lesssim \|f\|_{\overline{X}^{\mathcal L}_{\theta,\B_\eta,\widehat{E},a,F}}+\|f\|_{\overline{X}^{\mathcal R}_{\theta,\B_\eta,\widehat{E},a,F}}\quad \mbox{for all}\quad 0\leq\eta\leq 1.$$ 
%\overline{X}^{\mathcal R}_{\theta,\B_\eta,\widehat{E},a,F}\cap\overline{X}^{\mathcal L}_{\theta,\B_\eta,\widehat{E},a,F}\subseteq \overline{Y}_{\eta,\b,E}\quad \mbox{for all}\quad 0\leq\eta\leq 1. 
%$

Next, we will prove that $I_1+I_2\lesssim\min(\|f\|_{\overline{X}^{\mathcal R}_{\theta,\B_\eta,\widehat{E},a,F}},\|f\|_{\overline{X}^{\mathcal L}_{\theta,\B_\eta,\widehat{E},a,F}})$ when $0<\eta<1$.
The slowly varying function $$\psi(u):=\phi(u)^{-\eta} \b(\phi(u)),\qquad u>0,$$ satisfies that $\psi(u)\sim \psi(u^2)$ and has as associated functions, in the sense of \eqref{fb},
$$\Psi_0(v)=(\Phi_0(v))^{-\eta}\b(\Phi_0(v))\mand\Psi_\infty(v)=(\Phi_\infty(v))^{-\eta}\b(\Phi_\infty(v)),$$
for $0<v\leq1$. Using that $\b(\Phi_0(v))$ and $\b(\Phi_\infty(v))$ are slowly varying in $(0,1)$ and the properties of the extension indices, it is clear that
 $$\rho_{\Psi_0}\leq -\eta\pi_{\Phi_0}<0<-\eta\rho_{\Phi_\infty}\leq\pi_{\Psi_\infty}.$$
Hence, Lemma \ref{embeddings II} i) ($\pi_{B_{0,0}}>0$) and limiting Hardy type inequality i) ($\rho_{\Psi_0}<0<\pi_{\Psi_\infty}$) give
\begin{align*}
I_1&\lesssim  \Big \|\psi(u)\int_0^u\b_0(t) \|s^{- \theta} \a(s) K(s,f)\|_{ \widetilde{F}(t,u)} \varphi_{E_0}(\ell(t))\frac{dt}{t\ell(t)}\Big\|_{\widehat{E}}\\
&\leq \Big \|\psi(u)\int_0^u\b_0(t) \|s^{- \theta} \a(s) K(s,f)\|_{ \widetilde{F}(t,\infty)} \varphi_{E_0}(\ell(t))\frac{dt}{t\ell(t)}\Big\|_{\widehat{E}}\\
&\lesssim  \Big \|\psi(u)\b_0(u)\varphi_{E_0}(\ell(u)) \|s^{- \theta} \a(s) K(s,f)\|_{ \widetilde{F}(u,\infty)} \Big\|_{\widehat{E}}\\
&\sim \Big \|\B_{\eta}(u) \|s^{- \theta} \a(s) K(s,f)\|_{ \widetilde{F}(u,\infty)} \Big\|_{\widehat{E}}=\|f\|_{\overline{X}^{\mathcal R}_{\theta,\B_\eta,\widehat{E},a,F}}.
\end{align*}
Similarly, take the slowly varying function $\gamma(u):=\phi(u)^{1-\eta} \b(\phi(u))$, $u>0$,  which satisfies that $\gamma(u)\sim \gamma(u^2)$ and has as associated functions 
$$\Gamma_0(u)=(\Phi_0(u))^{1-\eta}\b(\Phi_0(u))\mand \Gamma_\infty(u)=(\Phi_\infty(u))^{1-\eta}\b(\Phi_\infty(u)),$$
for $0<u\leq 1$. The indices of these functions satisfy that
 $$\rho_{\Gamma_\infty}\leq (1-\eta)\rho_{\Phi_\infty}<0<(1-\eta)\pi_{\Phi_0}\leq\pi_{\Gamma_0}.$$
Then, Lemma \ref{embeddings II} ii) ($\pi_{B_{1,\infty}}>0$) and limiting Hardy type inequality ii) ($\rho_{\Gamma_\infty}<0<\pi_{\Gamma_0}$) can be applied to obtain that
\begin{align*}
I_2&\lesssim  \Big \|\gamma(u)\int_u^\infty\b_1(t) \|s^{- \theta} \a(s) K(s,f)\|_{ \widetilde{F}(u,t)} \varphi_{E_1}(\ell(t))\frac{dt}{t\ell(t)}\Big\|_{\widehat{E}}\\
&\leq \Big \|\gamma(u)\int_u^\infty\b_1(t) \|s^{- \theta} \a(s) K(s,f)\|_{ \widetilde{F}(0,t)} \varphi_{E_1}(\ell(t))\frac{dt}{t\ell(t)}\Big\|_{\widehat{E}}\\
&\lesssim  \Big \|\gamma(u)\b_1(u)\varphi_{E_1}(\ell(u)) \|s^{- \theta} \a(s) K(s,f)\|_{ \widetilde{F}(0,u)} \Big\|_{\widehat{E}}\\
&\sim \Big \|\B_\eta(u) \|s^{- \theta} \a(s) K(s,f)\|_{ \widetilde{F}(0,u)} \Big\|_{\widehat{E}}=\|f\|_{\overline{X}^{\mathcal L}_{\theta,\B_\eta,\widehat{E},a,F}}.
\end{align*}
Therefore 
\begin{equation}\label{emin}
I_1+I_2\lesssim\min(\|f\|_{\overline{X}^{\mathcal R}_{\theta,\B_\eta,\widehat{E},a,F}},\|f\|_{\overline{X}^{\mathcal L}_{\theta,\B_\eta,\widehat{E},a,F}}) \quad \mbox{for all}\quad 0<\eta< 1. 
\end{equation}
%and the union of the ${\mathcal R}$ and ${\mathcal L}$ spaces are included in the interpolation space. Namely,
%\begin{equation*}\label{ert}
%\overline{X}^{\mathcal R}_{\theta,\B_\eta,\widehat{E},a,F}\cup\overline{X}^{\mathcal L}_{\theta,\B_\eta,\widehat{E},a,F}\subseteq \overline{Y}_{\eta,\b,E},\qquad \mbox{for all}\quad  0<\eta<1.
%\end{equation*}

Secondly, we shall give separate arguments for the proof of the lower estimates of $I_1$ and $I_2$ for each of the cases. In case a) we are assuming that $\eta$ belongs to the interval $(0,M_1)$. By the definition of $M_1$, we have 
$$(1-\eta)[\rho_{\B_{0,\infty}}-\pi_{\varphi_{E_0}}]+\eta[\rho_{\B_{1,\infty}}-\pi_{\varphi_{E_1}}]<0<
(1-\eta)[\pi_{\B_{0,0}}-\rho_{\varphi_{E_0}}]+\eta[\pi_{\B_{1,0}}-\rho_{\varphi_{E_1}}]$$
and hence, the indices of the associated functions of the $\B_{\eta}$ satisfy that $\rho_{\B_{\eta,\infty}}<0<\pi_{\B_{\eta,0}}$. Applying  Lemma \ref{key1} ii) ($\rho_{B_{1,0}}<\pi_{\varphi_{E_1}}\leq \rho_{\varphi_{E_1}}<\pi_{B_{1,\infty}}$) we obtain that 
\begin{align}
\|f\|_{\overline{X}^{\mathcal R}_{\theta,\B_\eta,\widehat{E},a,F}}&=\Big\|\B_\eta(u)\|s^{- \theta} \a(s) K(s,f)\|_{ \widetilde{F}(u,\infty)}\Big\|_{\widehat{E}}\label{eI2bis}\\
&\lesssim\bigg \|\frac{\B_\eta(u)}{\|\b_1\|_{\widetilde{E}_1(u,\infty)}}\Big\| \b_1(t) \|s^{- \theta} \a(s) K(s,f)\|_{ \widetilde{F}(u,t)}\Big\|_{ \widetilde{E}_1(u,\infty)} \bigg\|_{\widehat{E}}\sim I_2.\nonumber
\end{align}
Summing up this estimate  with \eqref{ei12} and \eqref{emin} we complete the proof of a). 

In case b) we  are assuming the condition $M_2<\eta<1$.  By definition of $M_2$, we have 
$$(1-\eta)[\rho_{\B_{0,0}}-\pi_{\varphi_{E_0}}]+\eta[\rho_{\B_{1,0}}-\pi_{\varphi_{E_1}}]<0<
(1-\eta)[\pi_{\B_{0,\infty}}-\rho_{\varphi_{E_0}}]+\eta[\pi_{\B_{1,\infty}}-\rho_{\varphi_{E_1}}]$$ and then the function $\B_{\eta}$ satisfies that $\rho_{\B_{\eta,0}}<0<\pi_{\B_{\eta,\infty}}$. Applying  Lemma \ref{key1} i) ($\rho_{B_{0,\infty}}<\pi_{\varphi_{E_0}}\leq \rho_{\varphi_{E_0}}<\pi_{B_{0,0}}$) we deduce 
\begin{align}
\|f\|_{\overline{X}^{\mathcal L}_{\theta,\B_\eta,\widehat{E},a,F}}&=\Big\|\B_\eta(u)\|s^{- \theta} \a(s) K(s,f)\|_{ \widetilde{F}(0,u)}\Big\|_{\widehat{E}}\label{eI1bis}\\
&\lesssim\bigg \|\frac{\B_\eta(u)}{\|\b_0\|_{\widetilde{E}_0(0,u)}}\Big\| \b_0(t) \|s^{- \theta} \a(s) K(s,f)\|_{ \widetilde{F}(t,u)}\Big\|_{ \widetilde{E}_0(0,u)} \bigg\|_{\widehat{E}}\sim I_1.\nonumber
\end{align}
Therefore, using \eqref{ei12} and \eqref{emin} we show that $ \overline{Y}_{\eta,\b,E}=\overline{X}^{\mathcal L}_{\theta,\B_\eta,\widehat{E},a,F}$ if $\eta\in(M_2,1)$.

Next we proceed with the proof of c). Lemma \ref{maxI1I2} yields that 
$$I_1\sim\Big\|\Big(B_\eta(\lambda_k)
\|s^{- \theta} \a(s) K(s,f)\|_{\widetilde{F}(\lambda_{k-1},\lambda_k)} \Big)_{k\in\Z}\Big\|_{d\widehat{E}}$$
and 
$$I_2\lesssim \Big\|\Big(B_\eta(\lambda_k)
\|s^{- \theta} \a(s) K(s,f)\|_{\widetilde{F}(\lambda_{k-1},\lambda_k)} \Big)_{k\in\Z}\Big\|_{d\widehat{E}}$$
where $d\widehat{E}$ is the discretization of $E$.
Then, by \eqref{ei12}, it follows that
$$\|f\|_{\overline{Y}_{\eta,\b,E}}\sim\Big\|\Big(B_\eta(\lambda_k)
\|s^{- \theta} \a(s) K(s,f)\|_{\widetilde{F}(\lambda_{k-1},\lambda_k)} \Big)_{k\in\Z}\Big\|_{d\widehat{E}}.$$
Since $B_\eta(\lambda_k)\sim\psi(\lambda_k)\|\b_0\|_{\widetilde{E}_0(0,\lambda_k)}$, we deduce
$$
\|f\|_{\overline{Y}_{\eta,\b,E}}\sim\Big\|\Big(\psi(\lambda_k)\|\b_0\|_{\widetilde{E}_0(0,\lambda_k)}
\|s^{- \theta} \a(s) K(s,f)\|_{\widetilde{F}(\lambda_{k-1},\lambda_k)} \Big)_{k\in\Z}\Big\|_{d\widehat{E}}.$$
We note that 
$$\|\b_0\|_{\widetilde{E}_0(0,\lambda_k)}
\|s^{- \theta} \a(s) K(s,f)\|_{\widetilde{F}(\lambda_{k-1},\lambda_k)}\sim
\big\|s^{- \theta} \a(s)\|\b_0\|_{\widetilde{E}_0(0,s)} K(s,f)\big\|_{\widetilde{F}(\lambda_{k-1},\lambda_k)}.$$
Therefore, using Lemmas \ref{lema8} and \ref{lema29} ($\rho_{\Psi_0}<0<\pi_{\Psi_\infty}$) we obtain
\begin{align*}
\|f\|_{\overline{Y}_{\eta,\b,E}}&\sim\Big\|\Big(\psi(\lambda_k)
\big\|s^{- \theta} \a(s)\|\b_0\|_{\widetilde{E}_0(0,s)} K(s,f)\big\|_{\widetilde{F}(\lambda_{k-1},\lambda_k)} \Big)_{k\in\Z}\Big\|_{d\widehat{E}}\\
&\sim\Big\|\Big(\psi(\lambda_k)
\big\|s^{- \theta} \a(s)\|\b_0\|_{\widetilde{E}_0(0,s)} K(s,f)\big\|_{\widetilde{F}(0,\lambda_k)} \Big)_{k\in\Z}\Big\|_{d\widehat{E}}\\
&=\Big\|\sum_{k\in\Z}\psi(\lambda_k)
\big\|s^{- \theta} \a(s)\|\b_0\|_{\widetilde{E}_0(0,s)} K(s,f)\big\|_{\widetilde{F}(0,\lambda_k)} \chi_{I_k}(u)\Big\|_{\widehat{E}}\\
&\sim\Big\|\psi(u)
\big\|s^{- \theta} \a(s)\|\b_0\|_{\widetilde{E}_0(0,s)} K(s,f)\big\|_{\widetilde{F}(0,u)} \Big\|_{\widehat{E}}.
\end{align*}
This completes the proof of c).

Finally, we take care of the cases $\eta=0$ and $\eta=1$.
If $\eta=0$, it is clear that due to \eqref{eI2} and \eqref{eI2bis}, $I_1=\|f\|_{\overline{X}_{\theta,\b\circ\phi,\widehat{E},\b_0,E_0,\a,F}^{\mathcal R,\mathcal L}}$ while $I_2\sim\|f\|_{\overline{X}^{\mathcal R}_{\theta,\B_0,\widehat{E},a,F}}$. We have to observe that the arguments gave in \eqref{eI2bis} are also true if $\eta=0$.

If $\eta=1$, by \eqref{eI1} and \eqref{eI1bis} we have that $I_1\sim\|f\|_{\overline{X}^{\mathcal L}_{\theta,\B_1,\widehat{E},a,F}}$. Again we note that \eqref{eI1bis} is true in the case $\eta=1$. On the other hand, it is a fact that $I_2=\|f\|_{\overline{X}_{\theta,\b\circ\phi,\widehat{E},\b_1,E_1,\a,F}^{\mathcal L,\mathcal R}}$. The proof of the theorem is finished.

\end{proof}

\section{Applications}

For simplicity, we apply our results to ordered (quasi)-Banach
couples $\overline{X} = (X_0, X_1)$, in the sense that  $X_1\hookrightarrow X_0$. One of the most classical examples of an ordered couple is $(L_1(\Omega,\mu),L_\infty(\Omega,\mu))$ when $\Omega$ is a finite measure space.

\subsection{Ordered couples}\label{aordered}\hspace{2mm}

\vspace{2mm}
We briefly review how our definitions adapt to this simpler setting of ordered couples.
When $X_1\hookrightarrow X_0$, the real interpolation  $\overline{X}_{\theta,\b,E}$ can be equivalently defined as the space of all $f\in X_0$ such that 
\begin{equation}\label{e60}
\|f\|_{\theta,\b,E}=\|t^{-\theta}\b(t)K(t,f)\|_{\widetilde{E}(0,1)}<\infty,
\end{equation}
where $0\leq \theta\leq 1$, $E$ is an r.i. space on $(0,1)$ and $\b\in SV(0,1)$ (assuming the condition $\|\b\|_{\widetilde{F}(0,1)}<\infty$ if $\theta=1$).  
%\begin{rem}\label{rem61}
%\textup{We observe that 
%\begin{equation}\label{e61}
%\|f\|_{\theta,\b,E}\sim \|   t^{-\theta} \b(t) K(t,f) \|_{\widetilde{E}(0,\frac{1}{2})}.
%\end{equation}
%Indeed, using Lemma \ref{lem1} (i) and the fact that $t\varepsilont^{-1}K(t,\cdot)$ is non-increasing, we have
%\begin{align*}
%\|t^{-\theta} \b(t) K(t,f) \|_{\widetilde{E}(\frac{1}{2},1)}
%&\lesssim \|t^{1-\theta} \b(t)\|_{\widetilde{E}(\frac{1}{2},1)}K(\tfrac{1}{2},f)\lesssim \|t^{1-\theta} \b(t) \|_{\widetilde{E}(0,\frac{1}{2})}K(\tfrac{1}{2},f)\\
%&\lesssim\|t^{-\theta} \b(t) K(t,f) \|_{\widetilde{E}(0,\frac{1}{2})}.
%\end{align*}
%Then, equivalence \eqref{e61} is obtained if $0\leq \eta<1$. In the case $\theta=1$ with $\|\b\|_{\widtilde{F}(0,1)}<\infty$, Lemma \ref{lem1} (i), (iv) gives the result.}
%\end{rem}

Similarly, given a real parameter $0 \leq \theta \leq 1$,  $a, \b, c\in SV(0,1)$ and r.i. spaces $E, F, G$ on $(0,1)$, the spaces  
$\overline{X}_{\theta,\b,E,a,F}^{\mathcal L}$, $\overline{X}_{\theta,\b,\widehat{E},a,F}^{\mathcal L}$ and $\overline{X}_{\theta,c,\widehat{E},\b,F,a,G}^{\mathcal R,\mathcal L}$ are defined just as in Definitions \ref{defLR} and \ref{defLRR^}; the only change being that $\widetilde{E}(0,\infty)$ must be replaced by $\widetilde{E}(0, 1)$ and $\widehat{E}(0,\infty)$ by $\widehat{E}(0,1)$.
Likewise, the spaces  $\overline{X}_{\theta,\b,E,a,F}^{\mathcal R}$ and $\overline{X}_{\theta,c,\widehat{E},\b,F,a,G}^{\mathcal L,\mathcal R}$ are defined as
$$\overline{X}_{\theta,\b,E,a,F}^{\mathcal R}=\Big\{f\in X_0:\big \|  \b(t) \|   s^{-\theta} a(s) K(s,f) \|_{\widetilde{F}(t,1)}      \big   \|_{\widetilde{E}(0,1)}<\infty\Big\}$$
and
$$\overline{X}_{\theta,c,\widehat{E},\b,F,a,G}^{\mathcal L,\mathcal R }=\bigg\{f\in X_0: 
\Big\|c(u)\big\|\b(t) \|s^{-\theta} \a(s) K(s,f) \|_{\widetilde{G}(u,t)}\big\|_{\widetilde{F}(u,1)}\Big\|_{\widehat{E}(0,1)} < \infty\bigg\}.$$
The space $\overline{X}_{\theta,\b,\widehat{E},a,F}^{\mathcal R}$ can be defined analogously replacing $\widetilde{E}$ by $\widehat{E}$.

If $\overline{X}$ is an ordered couple, then $(\overline{X}^{\mathcal R}_{\theta, \b_0,E_0,\a,F}, \overline{X}^{\mathcal L}_{\theta,\b_1,E_1,\a,F})$ is also ordered as we proved in the following lemma.
\begin{lem}\label{lemainclusion}
Let $\overline{X}$ be an ordered (quasi-) Banach couple,  $E_0,\ E_1,\ F$ r.i. spaces on $(0,1)$ and $a, \b_0, \b_1\in SV(0,1)$ such that $\|\b_0\|_{\widetilde{E}_0(0,1)}<\infty$. If $0\leq\theta<1$ or $\theta=1$ and $\|a\|_{\widetilde{F}(0,1)}<\infty$, then
$$ \overline{X}^{\mathcal L}_{\theta,\b_1,E_1,a,F}\hookrightarrow\overline{X}_{\theta,a,F}\hookrightarrow \overline{X}^{\mathcal R}_{\theta,\b_0,E_0,a,F}.$$
\end{lem}

\begin{proof}
The first embedding is an easy consequence of the fact that interval $(0,1)$ in \eqref{e60} can be reduce to $(0,1/2)$. Indeed,
\begin{align*}
\|f\|_{{\mathcal L};\theta,\b_1,E_1,a,F}&\geq \big \|  \b_1(t) \|   s^{-\theta} a(s) K(s,f) \|_{\widetilde{F}(0,t)}      \big   \|_{\widetilde{E}_1(\tfrac{1}{2},1)}\\ &\gtrsim  \|\b_1\|_{\widetilde{E}_1(\tfrac{1}{2},1)}\|s^{-\theta} a(s) K(s,f) \|_{\widetilde{F}(0,\tfrac{1}{2})}\sim\|f\|_{\theta,a,F}.
\end{align*}
The second one follows directly using the definition of the norm in the $\mathcal{R}$-space. 
%\begin{align*}
%\|f\|_{{\mathcal R};\theta,\b_0,E_0,a,F}&= \big \|  \b_0(t) \|   s^{-\theta} a(s) K(s,f) \|_{\widetilde{F}(t,1)}      \big   \|_{\widetilde{E}_0(0,1)}\\ &\leq   \|\b_0\|_{\widetilde{E}_0(0,1)}\|s^{-\theta} a(s) K(s,f) \|_{\widetilde{F}(0,1)}\sim\|f\|_{\theta,a,F}.
%\end{align*}
%To prove the first one, we identify the $\mathcal{L}$-space, using Theorem 5.3 from \cite{FMS-2}, as an interpolation space of an ordered couple. Indeed,
%$$\overline{X}^{\mathcal L}_{\theta,\b_1,E_1,a,F}=(\overline{X}_{\theta,\a,F},\overline{X}_{1,\b_0,E_0})_{0,\b_1\circ \sigma^{-1},E}$$
%where $\sigma$ is a strongly increasing, differentiable function such that $\sigma(t)\sim t^{1-\theta}\frac{a(t)}{\|\b_0\|_{\widetilde{E}_0(0,t)}}$, $t\in(0,1)$ (see Remark 3 from \cite{Do2020}). Due to Lemma 5.2 from \cite{FMS-RL1} the couple $(\overline{X}_{\theta,\a,F},\overline{X}_{1,\b_0,E_0})$ is ordered, and hence 
%$$(\overline{X}_{\theta,\a,F},\overline{X}_{1,\b_0,E_0})_{0,\b_1\circ \sigma^{-1},E}\hookrightarrow \overline{X}_{\theta,\a,F}$$
%that gives the desired embedding.
\end{proof}

Of course, the results of the previous sections remain true if we work with slowly varying functions on $(0,1)$, r.i. spaces on $(0,1)$ and ordered couples. In these cases all assumptions concerning the interval $(1,\infty)$  must be omitted. For example, Lemma \ref{lema2.3} reads as follows:

\begin{lem}
Let $E$ be an r.i. space on $(0,1)$ and let $\b\in SV(0,1)$ such that $\b(t^{2}) \sim \b(t)$, with associated function $\B_0$ defined in (\ref{fb}).
\begin{itemize}
\item[i)] If $\rho_{\varphi_E} < \pi_{\B_{0}}$, then 
$$ \|  \b \|_{\widetilde{E}(0,t)} \sim \b(t) \varphi_{E}(\ell(t)),\qquad t\in(0,1).$$
\item[ii)] If $\rho_{\B_{0}} < \pi_{\varphi_E}$, then 
the equivalence is
$$\| \b \|_{\widetilde{E}(t,1)} \sim \b(t)  \varphi_{E} (\ell(t)),\qquad t\in(0,1/2).$$
\end{itemize}
\end{lem}
That's reduction to the interval $(0,1/2)$ is not a problem since if $a$ and $\b$ are two slowly varying functions such that $\b(t)\sim\a(t)$ for all $t\in(0,1/2)$ then
$\overline{X}_{\theta,\b,E}=\overline{X}_{\theta,\a,E}.$
A similar identity holds for $\mathcal{R}$, $\mathcal{L}$ spaces and for the extreme constructions $\mathcal{R,L}$ and $\mathcal{L,R}$.

Moreover, if the couple $\overline{X}$ is ordered Theorem \ref{maintheorem} reads as follows:
\begin{thm} \label{maintheorem2}
Let $0<\theta<1$. Let  $E$, $E_0$, $E_1$, $F$ be  r.i. spaces and let $\a$, $\b$,  $\b_0$, $\b_1\in SV(0,1)$ be  such that $\b_0(t)\sim\b_0(t^2)$ and $\b_1(t)\sim\b_1(t^2)$.  Assume that $E_0$, $E_1$ and  the associated functions $B_{0,0}$, $B_{1,0}$ of $\b_0$ and $\b_1$, respectively, satisfy that $$\rho_{\varphi_{E_0}}<\pi_{B_{0,0}}\quad\mbox{and}\quad 
\rho_{B_{1,0}}<\pi_{\varphi_{E_1}}.$$ Let $0\leq\eta\leq1$ be a parameter and define
$$B_\eta(u)=\big(\b_0(u)\varphi_{E_0}(\ell(u))\big)^{1-\eta}\;\big(\b_1(u)\varphi_{E_1}(\ell(u))\big)^{\eta}\;\b\Big(\frac{\b_0(u)\varphi_{E_0}(\ell(u))}{\b_1(u)\varphi_{E_1}(\ell(u))}\Big),\quad u\in(0,1).$$
\begin{itemize}
\item [a)] If $0\leq\eta<M_1:=\Big(1-\frac{\pi_{\B_{1,0}}-\rho_{\varphi_{E_1}}}{\pi_{\B_{0,0}}-\rho_{\varphi_{E_0}}}\Big)^{-1}$, then
$$\left(\overline{X}^{\mathcal R}_{\theta, \b_0,E_0,\a,F}, \overline{X}^{\mathcal L}_{\theta,\b_1,E_1,\a,F}\right)_{\eta,\b,E}=\overline{X}^{\mathcal R}_{\theta,\B_\eta,\widehat{E},a,F}.$$
\item [b)] If $M_2:=\Big(1-\frac{\rho_{\B_{1,0}}-\pi_{\varphi_{E_1}}}{\rho_{\B_{0,0}}-\pi_{\varphi_{E_0}}}\Big)^{-1}<\eta\leq1$, then
$$\left(\overline{X}^{\mathcal R}_{\theta, \b_0,E_0,\a,F}, \overline{X}^{\mathcal L}_{\theta,\b_1,E_1,\a,F}\right)_{\eta,\b,E}=\overline{X}^{\mathcal L}_{\theta,\B_\eta,\widehat{E},a,F}.$$
\item[c)] If $M_1\leq \eta\leq M_2$, then
$$\left(\overline{X}^{\mathcal R}_{\theta, \b_0,E_0,\a,F}, \overline{X}^{\mathcal L}_{\theta,\b_1,E_1,\a,F}\right)_{\eta,\b,E}= \overline{X}^{\mathcal L}_{\theta,\B_\eta^\#,\widehat{E},\a^\#,F}$$
where $\B_\eta^\#(u)=\frac{\B_\eta(u)}{\b_0(u)\varphi_{E_0}(\ell(u))}$ and 
$\a^\#(u)=\a(u)\b_0(u)\varphi_{E_0}(\ell(u))$, $u\in(0,1)$.
\item[d)] If $\eta=0$, then
$$\left(\overline{X}^{\mathcal R}_{\theta, \b_0,E_0,\a,F}, \overline{X}^{\mathcal L}_{\theta,\b_1,E_1,\a,F}\right)_{0,\b,E}= \overline{X}^{\mathcal R}_{\theta,\B_0,\widehat{E},a,F}\cap\overline{X}_{\theta,\b\circ\phi,\widehat{E},\b_0,E_0,\a,F}^{\mathcal R,\mathcal L}$$
where  $\phi(u)=\frac{\b_0(u)\varphi_{E_0}(\ell(u))}{\b_1(u)\varphi_{E_1}(\ell(u))}$, $u\in(0,1)$.

\item[e)] If $\eta=1$ and $\|\b\|_{\widetilde{E}(0,1)}<\infty$, then
$$\left(\overline{X}^{\mathcal R}_{\theta, \b_0,E_0,\a,F}, \overline{X}^{\mathcal L}_{\theta,\b_1,E_1,\a,F}\right)_{1,\b,E}= \overline{X}^{\mathcal L}_{\theta,\B_1,\widehat{E},a,F}\cap\overline{X}_{\theta,\b\circ\phi,\widehat{E},\b_1,E_1,\a,F}^{\mathcal L,\mathcal R}$$
where $\phi$ is defined in d).
\end{itemize}
\end{thm}

\subsection{Interpolation between grand and small Lebesgue spaces}

Here, we present interpolation identities for the grand and small Lebesgue spaces as application of our reiteration theorem.

The identity 
$$(L_1,L_\infty)_{1-\frac1p,\b,E}=L_{p,\b,E}$$
for $1<p\leq \infty$, follows from
Peetre's  formulal for the $K$-functional 
$$K(t,f;L_1,L_\infty)=\int_0^tf^*(s)\, ds= tf^{**}(t),\quad t>0,$$
and the equivalence $\|t^{1/p}\b(t)f^{**}(t)\|_{\widetilde{E}}\sim \|t^{1/p}\b(t)f^*(t)\|_{\widetilde{E}}$, $1<p\leq \infty$, $\b\in SV$ (see, e.g. \cite[Lemma 2.16]{CwP}).
Analogously, it can be proved that
$$(L_1,L_\infty)^{\mathcal R}_{1-\frac1p,\b,\widehat{E},a,F}=L^{\mathcal R}_{p,\b,\widehat{E},a,F}\mand (L_1,L_\infty)^{\mathcal L}_{1-\frac1p,\b,\widehat{E},a,F}=L^{\mathcal L}_{p,\b,\widehat{E},a,F}$$
where
$$L^{\mathcal R}_{p,\b,\widehat{E},a,F}:=\Big\{f\in\mathcal{M}(\Omega,\mu):
\Big\|\b(t) \|s^{1/p} a(s) f^*(s)\|_{\widetilde{F}(t,1)}\Big\|_{\widehat{E}(0,1)} < \infty\Big\}$$
and 
$$L^{\mathcal L}_{p,\b,\widehat{E},a,F}:=\Big\{f\in\mathcal{M}(\Omega,\mu):
\Big\|\b(t) \|s^{1/p} a(s) f^*(s)\|_{\widetilde{F}(0,t)}\Big\|_{\widehat{E}(0,1)} < \infty\Big\}$$
for $E$, $F$ r.i. spaces, $a$, $\b\in SV$ and $1<p\leq \infty$. In a similar vein
$$(L_1,L_\infty)^{\mathcal R,\mathcal L}_{1-\frac1p,c,\widehat{E},\b,F,a,G}=L^{\mathcal R,\mathcal L}_{p,c,\widehat{E},\b,F,a,G}\mand (L_1,L_\infty)^{\mathcal L,\mathcal R}_{1-\frac1p,c,\widehat{E},\b,F,a,G}=L^{\mathcal L,\mathcal R}_{p,c,\widehat{E},\b,F,a,G}$$
where the spaces $L^{\mathcal R,\mathcal L}_{p,c,\widehat{E},\b,F,a,G}$ and $L^{\mathcal L,\mathcal R}_{p,c,\widehat{E},\b,F,a,G}$ are defined as the set of all $f\in\mathcal{M}(\Omega,\mu)$ for which \eqref{dRL}-\eqref{dLR} are satisfied after the change of $s^{-\theta}a(s)K(s,f)$ for $s^{1/p}a(s)f^*(s)$, $s>0$, respectively.

We follow the paper by Fiorenza and Karadzhov \cite{FK} in order to give the next definition:
\begin{defn}\label{def_gLp}
Let $(\Omega,\mu)$ be a finite measure space such that $\mu(\Omega)=1$, let $1<p<\infty$ and $\alpha>0$. The grand Lebesgue space $L^{p),\alpha}$ is  the set of all $f\in{\mathcal M}(\Omega,\mu)$ such that
\begin{equation*}\label{lpa1}
\|f\|_{p),\alpha}=\Big\|\ell^{-\frac{\alpha}{p}}(t)\|s^{1/p}f^{*}(s)\|_{\widetilde{L}_p(t,1)}\Big\|_{L_\infty(0,1)}<\infty.
\end{equation*}
The small Lebesgue space $L^{(p,\alpha}(\Omega)$ is  the set of all $f\in{\mathcal M}(\Omega,\mu)$ such that
\begin{equation*}\label{lpa2}
\|f\|_{(p,\alpha}=\Big\|\ell^{\frac{\alpha}{p'}-1}(t)\|s^{1/p}f^{*}(s)\|_{\widetilde{L}_p(0,t)}\Big\|_{\widetilde{L}_1(0,1)}<\infty
\end{equation*}
where $\frac1p+\frac1{p'}=1$.
\end{defn}

The classical grand Lebesgue space $L^{p)}(\Omega):=L^{p),1}(\Omega)$ was introduced by Iwaniec and Sbordone in \cite{IS} while the classical small Lebesgue space $L^{(p}(\Omega):=L^{(p,1}(\Omega)$ was characterized by Fiorenza  in \cite{F1} as its associate; that is $ (L^{(p'})' = L^{p)}.$ For more information about this spaces and their generalizations see the recent paper \cite{FFG}.

As observed in \cite{FK,op1}, these spaces can be characterized as $\mathcal{R}$ and  $\mathcal{L}$-spaces in the following way
$$L^{p),\alpha}=L^{\mathcal R}_{p,\ell^{-\alpha/p}(u),L_\infty,1,L_p}\mand
L^{(p,\alpha}=L^{\mathcal L}_{p,\ell^{\alpha/p'-1}(u),L_1,1,L_p}.$$
Thus, we can apply the main theorem of this paper in order to identify the interpolation space $(L^{p),\alpha},L^{(p,\beta})_{\eta,\b,E}$. 
Next result recovers Theorem 6.2 form \cite{FFGKR}, and additionally completes it with the case $\alpha\neq\beta$ and with the extreme cases $\eta=0,1$.

\begin{cor}\label{cor57}
Let $E$ be an r.i. space on $(0,1)$, $\b\in SV(0,1)$, $1< p<\infty$ and $\alpha, \beta>0$. Let $0\leq \eta\leq 1$ a parameter and define
$\B_\eta(u)=\ell^{-\frac{\alpha(1-\eta)}{p}+\frac{\beta\eta}{p'}}(u)\,\b\big(\ell^{\frac{\beta-\alpha}{p}-\beta}(u)\big)$, $u \in(0,1)$. The following statements hold:
\begin{itemize}
\vspace{1mm}
\item[a)] If $0<\eta< \frac{\alpha}{\alpha-\beta+p\beta}$, then
\begin{equation*}
\big(L^{p),\alpha},L^{(p,\beta}\big)_{\eta,\b,E}=L^{\mathcal{R}}_{p,\B_\eta,\widehat{E},1,L_p}.
\end{equation*}
\item[b)] If $\frac{\alpha}{\alpha-\beta+p\beta}<\eta< 1$, then
\begin{equation*}
\big(L^{p),\alpha},L^{(p,\beta}\big)_{\eta,\b,E}=L^{\mathcal{L}}_{p,\B_\eta,\widehat{E},1,L_p}.
\end{equation*}
\item[c)] If $\eta=\frac{\alpha}{\alpha-\beta+p\beta}$, then
$$
\big(L^{p),\alpha},L^{(p,\beta}\big)_{\eta,\b,E}=L^{\mathcal L}_{p,B_\eta^{\#},\widehat{E},\ell^{-\alpha/p}(u),L_p}$$
where $\B_\eta^{\#}(u)=\ell^{\eta(\frac{\alpha-\beta}{p}+\beta)}(u)\,\b\big(\ell^{\frac{\beta-\alpha}{p}-\beta}(u)\big)$, $u \in(0,1)$. 

\item[d)] If $\eta=0$, then
$$\big(L^{p),\alpha},L^{(p,\beta}\big)_{0,\b,E}=L^{\mathcal R}_{p,B_0,\widehat{E},1,L_p}\cap L_{p,\b\circ\phi,\widehat{E},\ell^{-\alpha/p}(u),L_\infty,1,L_p}^{\mathcal R,\mathcal L}$$
where $\phi(u)=\ell^{\frac{\beta-\alpha}{p}-\beta}(u)$, $u\in(0,1)$.

\item[e)] If $\eta=1$ and $\|\b\|_{\widetilde{E}(0,1)}<\infty$, then
$$\big(L^{p),\alpha},L^{(p,\beta}\big)_{1,\b,E}=L^{\mathcal L}_{p,B_1,\widehat{E},1,L_p}\cap L_{p,\b\circ\phi,\widehat{E},\ell^{\beta/p'-1}(u),L_1,1,L_p}^{\mathcal L,\mathcal R}$$
where $\phi$ is defined in d).
\end{itemize}
\end{cor}

\begin{proof}
Denote $\b_0(u)=\ell^{-\alpha/p}(u)$, $\b_1(u)=\ell^{\beta/p'-1}(u)$, $u\in(0,1)$, 
$E_0=L_\infty$ and $E_1=L_1$. Then, it is clear that $\pi_{B_{0,0}}=\rho_{B_{0,0}}=-\frac{\alpha}{p}$, 
$\pi_{B_{1,0}}=\rho_{B_{1,0}}=\frac{\beta}{p'}-1$, $\pi_{\varphi_{E_0}}=\rho_{\varphi_{E_0}}=0$ and that $\pi_{\varphi_{E_0}}=\rho_{\varphi_{E_0}}=1$. Hence, $M_1=M_2=\frac{\alpha}{\alpha-\beta+p\beta}$. 
Moreover, 
$$\b_0(u)\varphi_{E_0}(u)=\ell^{-\alpha/p}(u)\mand \b_1(u)\varphi_{E_1}(u)=\ell^{\beta/p'}(u),\quad u\in(0,1).$$
Thus, the result follows by direct  application of Theorem \ref{maintheorem2}. 
\end{proof}

{\small \hspace{-5mm}{\textbf{Acknowledgments.}}
Second and third authors have been partially supported by grant MTM2017-84058-P (AEI/FEDER, UE).
The third author also thanks the Isaac Newton Institute for Mathematical Sciences, Cambridge, for support and hospitality during the programme \emph{Approximation, Sampling and Compression in Data Science} where the work on this paper was undertaken; this work was supported by EPSRC grant no. EP/R014604/1. Finally, the third author thanks \'Oscar Dom\'{\i}nguez for useful conversations at the early stages of this work, and for pointing out the reference \cite{FFGKR}.}

\end{document}